\documentclass[11pt, reqno]{amsart}

\usepackage[margin=0.9in]{geometry}
\usepackage[utf8]{inputenc}
\usepackage[T1]{fontenc}

\usepackage{subfiles}
\usepackage{comment}
\usepackage{float}
\usepackage{marginnote}
\usepackage{euscript}
\usepackage[dvipsnames]{xcolor}   		             		
\usepackage{graphicx}			
\usepackage{amssymb}
\usepackage{mathrsfs}
\usepackage{amsthm}
\usepackage{amsmath, amsrefs}
\usepackage[foot]{amsaddr}
\usepackage{mathtools}
\usepackage[cal=cm]{mathalfa}
\usepackage{stmaryrd}
\usepackage{upgreek}
\usepackage{bbm}
\usepackage{xypic, dynkin-diagrams}
\usepackage[hypertexnames=false, hyperfootnotes=false, colorlinks=true,linktocpage=true, allcolors=highlight, bookmarksdepth = 2]{hyperref}
\usepackage{cleveref}
\usepackage{url}
\usepackage{float}
\usepackage[full]{textcomp}
\makeatletter
\newcommand{\citecomment}[2][]{\citen{#2}#1\citevar}
\newcommand{\citeone}[1]{\citecomment{#1}}
\newcommand{\citetwo}[2][]{\citecomment[,~#1]{#2}}
\newcommand{\citevar}{\@ifnextchar\bgroup{;~\citeone}{\@ifnextchar[{;~\citetwo}{]}}}
\newcommand{\citefirst}{\@ifnextchar\bgroup{\citeone}{\@ifnextchar[{\citetwo}{]}}}

\makeatother

\usepackage{tikz,tikz-cd,tikz-3dplot}
\usepackage{pgfplots}
\pgfplotsset{compat=1.18}
\usetikzlibrary{calc}
\usetikzlibrary{fadings}
\usetikzlibrary{decorations.pathmorphing}
\usetikzlibrary{decorations.pathreplacing}
\usetikzlibrary{patterns}
\usetikzlibrary{arrows,shadows,positioning, calc, decorations.markings, 
	hobby,quotes,angles,decorations.pathreplacing,intersections,shapes}
\usepgflibrary{shapes.geometric}
\usetikzlibrary{fillbetween,backgrounds}

\usepackage{xcolor}

\definecolor{highlight}{HTML}{3465a4}

\usepackage{anyfontsize}
\usepackage[lining]{libertine}%
\usepackage{courier}
\usepackage{csquotes}
\makeatletter
\renewcommand*\libertine@figurestyle{LF}
\makeatother
\usepackage[libertine,libaltvw,liby]{newtxmath}
\usepackage{microtype}

\usepackage{array}
\newcolumntype{H}{>{\setbox0=\hbox\bgroup}c<{\egroup}@{}}

\BibSpec{article}{%
+{}{\sc \PrintAuthors} {author}
+{,}{ } {title}
+{,}{ \textit} {journal}
+{,}{ } {volume}
+{}{ \IfEmptyBibField{journal}{}{\PrintDate{date}}} {transition}
+{,}{ } {pages}
+{,}{ } {status}
+{.}{} {transition}
+{}{ \, Preprint available at \tt} {eprint}
+{.}{} {transition}
}

\BibSpec{book}{%
+{}{\sc \PrintAuthors} {author}
+{,}{ } {title}
+{.}{ \textit} {series}
+{,}{ } {volume}
+{.}{ } {publisher}
+{,}{ } {place}
+{,}{ } {date}
+{.}{} {transition}
}

\BibSpec{incollection}{
+{} {\sc \PrintAuthors} {author}
+{,} { } {title}
+{,} { in \textit} {booktitle}
+{,}{ } {series}
+{,}{ } {volume}
+{}{ \PrintDate } {date}
+{,} { pp.~} {pages}
+{,}{ } {publisher}
+{ }{ Preprint available at \tt} {eprint}
+{.}{} {transition}
}

\BibSpec{inproceedings}{
+{} {\sc \PrintAuthors} {author}
+{,} { } {title}
+{,} { in \textit} {booktitle}
+{,} { pp.~} {pages}
+{,} { \PrintDateB} {date}
+{,} { available at \eprint} {eprint}
+{.} {} {transition}
}

\BibSpec{collection.article}{
+{} {\sc \PrintAuthors} {author}
+{,} { } {title}
+{,} { in \it \PrintConference} {conference}
+{,} { pp.~} {pages}
+{.} {\PrintBook} {book}
+{,} { \PrintDateB} {date}
+{.} {} {transition}
}

\BibSpec{innerbook}{
+{.} { \emph} {title}
+{.} { } {part}
+{:} { \emph} {subtitle}
+{.} { } {series}
+{,} { } {volume}
+{.} { Edited by \PrintNameList} {editor}
+{.} { Translated by \PrintNameList}{translator}
+{.} { \PrintContributions} {contribution}
+{.} { } {publisher}
+{.} { } {organization}
+{,} { } {address}
+{,} { \PrintEdition} {edition}
+{,} { \PrintDateB} {date}
+{.} { } {note}
}

\tikzset{
	commutative diagrams/.cd, 
	arrow style=tikz, 
	diagrams={>=stealth}
}
\tikzset{
	arrow/.pic={\path[tips,every arrow/.try,->,>=#1] (0,0) -- +(0,4pt);},
	pics/arrow/.default={triangle 90}
}
\tikzset{->-/.style={decoration={
			markings,
			mark=at position .6 with {\arrow{latex}}},postaction={decorate}}
}
\tikzset{
	c/.style={every coordinate/.try}
}
\usepackage{enumerate}

\newcommand{\mf}{\mathfrak}

\DeclareMathOperator*{\Res}{Res}

\usepackage{xifthen}
\newcommand{\Gm}[1][]{%
	\ifthenelse{\isempty{#1}}%
	{\mathbb{C}^\times}
	{(\mathbb{C}^\times)^{#1}}
}

\newcommand{\intEquiv}[1]{{}^{#1}\!\!\!\int}

\newcommand{\rd}{\mathrm{d}}
\newcommand{\ri}{\mathrm{i}}
\newcommand{\re}{\mathrm{e}}

\newcommand{\de}{{\partial}}

\newcommand{\bbV}{\mathbb{V}}

\newcommand{\bbZ}{\mathbb{Z}}

\newcommand{\bbC}{\mathbb{C}}
\newcommand{\bbP}{\mathbb{P}}

\newcommand{\bbQ}{\mathbb{Q}}

\newcommand{\cO}{\mathcal{O}}
\newcommand{\cT}{\mathcal{T}}

\newcommand{\cP}{\mathcal{P}}
\newcommand{\cC}{\mathcal{C}}

\newcommand{\cK}{\mathcal{K}}

\newcommand{\cW}{\mathcal{W}}

\newcommand{\cA}{\mathcal{A}}
\newcommand{\HH}{\mathcal{H}}

\newcommand{\cI}{\mathcal{I}}
\newcommand{\cJ}{\mathcal{J}}
\newcommand{\RR}{\mathcal{R}}
\newcommand{\cX}{\mathcal{X}}

\newcommand{\cM}{\mathcal M}
\newcommand{\cQ}{\mathcal Q}

\renewcommand{\l}{\left}
\renewcommand{\r}{\right}

\def\beq{\begin{equation}}                     %
	\def\eeq{\end{equation}}                       %
\def\bea{\begin{eqnarray}}                     
	\def\eea{\end{eqnarray}}
\def\bary{\begin{array}} 
	\def\eary{\end{array}} 
\def\ben{\begin{enumerate}} 
	\def\een{\end{enumerate}}
\def\bit{\begin{itemize}} 
	\def\eit{\end{itemize}}
\def\nn{\nonumber}

\theoremstyle{plain}
\newtheorem{thm}{Theorem}[section]
\newtheorem*{thm*}{Theorem}
\newtheorem{lem}[thm]{Lemma}

\newtheorem{prop}[thm]{Proposition}
\newtheorem*{prop*}{Proposition}

\newtheorem*{conj*}{Conjecture}

\newtheorem{cor}[thm]{Corollary}
\newtheorem*{cor*}{Corollary}

\theoremstyle{definition}
\newtheorem{defn}[thm]{Definition}
\newtheorem{rmk}[thm]{Remark}

\newtheorem{example}[thm]{Example}

\theoremstyle{plain}

\theoremstyle{plain}

\theoremstyle{plain}

\theoremstyle{definition}

\theoremstyle{plain}

\crefname{equation}{Eq.}{Eqs.}
\crefname{eqnarray}{Eq.}{Eqs.}
\crefname{algo}{algorithm}{algorithms}
\crefname{conj}{conjecture}{conjectures}
\crefname{lem}{lemma}{lemmas}
\crefname{thm}{theorem}{theorems}
\crefname{claim}{claim}{claims}
\crefname{rmk}{remark}{remarks}
\crefname{prop}{proposition}{propositions}
\crefname{section}{section}{sections}
\crefname{appendix}{appendix}{appendices}
\crefname{cor}{corollary}{corollaries}
\crefname{figure}{figure}{figures}
\crefname{table}{table}{tables}
\crefname{example}{example}{examples}
\crefname{prob}{problem}{problems}
\crefname{assm}{assumption}{assumptions}
\crefname{defn}{definition}{definitions}
\crefname{notation}{notation}{notations}
\crefname{speculation}{speculation}{speculations}
\crefname{construction}{construction}{constructions}
\crefname{observation}{observation}{observations}
\crefname{innercustomthm}{Theorem}{Theorems}
\crefname{innercustomconj}{Conjecture}{Conjectures}
\crefname{innercustomassumption}{assumption}{Assumption}
\crefname{innerpracticalresult}{practical result}{practical results}

\newcommand{\bra}{\left\langle}
\newcommand{\ket}{\right\rangle}

\crefname{equation}{Eq.}{Eqs.}
\crefname{eqnarray}{Eq.}{Eqs.}
\crefname{algo}{Algorithm}{Algorithms}
\crefname{conj}{Conjecture}{Conjectures}
\crefname{lem}{Lemma}{Lemmas}
\crefname{thm}{Theorem}{Theorems}
\crefname{customthm}{Theorem}{Theorems}
\crefname{claim}{Claim}{Claims}
\crefname{rmk}{Remark}{Remarks}
\crefname{prop}{Proposition}{Propositions}
\crefname{section}{Section}{Sections}
\crefname{appendix}{Appendix}{Appendices}
\crefname{cor}{Corollary}{Corollaries}
\crefname{figure}{Figure}{Figures}
\crefname{table}{Table}{Tables}
\crefname{example}{Example}{Examples}
\crefname{prob}{Problem}{Problems}
\crefname{assm}{Assumption}{Assumptions}
\crefname{defn}{Definition}{Definitions}
\crefname{customconj}{Conjecture}{Conjectures}

\numberwithin{equation}{section}

\title{Dubrovin duality and mirror symmetry for ADE resolutions}
\author{Andrea Brini$^{1,2}$, Jingxiang Ma$^1$, Ian A. B. Strachan$^3$}
\address{$^1$ University of Sheffield, School of Mathematical and Physical Sciences, Hounsfield Road, Sheffield S3 7RH, United Kingdom}
\address{$^2$ On leave from CNRS, DR~13, Montpellier, France}
\address{$^3$ University of Glasgow, School of Mathematics and Statistics, 132 University Pl, Glasgow G12 8TA, United Kingdom}
\email{a.brini@sheffield.ac.uk,jma75@sheffield.ac.uk,ian.strachan@glasgow.ac.uk}

\allowdisplaybreaks
\begin{document}

\setcounter{tocdepth}{2}

\begin{abstract}
We show that, under Dubrovin's notion of ``almost'' duality, the Frobenius manifold structure on the orbit spaces
of the extended affine Weyl groups of type $\mathrm{ADE}$ is dual, for suitable choices of weight markings, 
to the equivariant quantum cohomology of the minimal resolution of the du~Val singularity of the same Dynkin type. 
We also provide a uniform Lie-theoretic construction of Landau--Ginzburg mirrors for the quantum cohomology of $\mathrm{ADE}$ resolutions. The mirror B-model is described by a one-dimensional LG superpotential associated to the spectral curve of the $\widehat{\mathrm{ADE}}$ affine relativistic Toda chain.
\end{abstract}
\maketitle

\section{Introduction} 

Let $l \in \bbZ_{>0}$ and $\mathcal{R}=\mathrm{ADE}_l$ be a rank$-l$ simply-laced irreducible root system, and  
fix a choice of fundamental weight $\widehat{\omega}$ for the complex simple Lie algebra associated to $\RR$ as follows: 
\bit
\item when $\mathcal{R}=\mathrm{A}_l$, $\widehat{\omega}$ can be any fundamental weight;
\item when $\mathcal{R}=\mathrm{D}_l$ or $\mathcal{R}=\mathrm{E}_l$, $\widehat{\omega}$ will be the highest weight of the fundamental representation of highest dimension.
\eit
We will call the datum $(\RR, \widehat{\omega})$ a \emph{marked ADE pair}. The corresponding Dynkin diagrams, with node marking specified by $\widehat{\omega}$, are shown in \cref{fig:dynkin}. In this paper we will be concerned with three classes of Frobenius manifolds 
associated to a given pair $(\RR, \widehat{\omega})$, arising respectively from representation theory, enumerative algebraic geometry, and integrable systems.

\begin{itemize}
    \item For $\RR$ the root system of any complex simple Lie algebra, and $\widehat{\omega} \in \Omega$ any fundamental weight,  
    Dubrovin and Zhang \cite{MR1606165} famously
    constructed
a canonical semi-simple Frobenius manifold structure $\mathrm{EAW}(\RR, \widehat{\omega})$
on the orbits of the reflection representation of the $\widehat{\omega}$-extended affine Weyl group of $\RR$,
generalising the classical construction of polynomial Frobenius manifolds on orbit spaces of Coxeter groups. The specialisation to $(\RR,\widehat{\omega})$ being a marked ADE pair will be the setup of sole concern to us in this paper, and we will use the shorthand notation
\[
\cM_{\rm AW} \coloneqq
\mathrm{EAW}(\RR, \widehat{\omega})\,.\]
    
\item Let $\mathsf{G}<\mathrm{SL}(2,\bbC)$, $|\mathsf{G}|<\infty$ be the McKay group of type $\RR$, and let $Z=\widetilde{\bbC^2/\mathsf{G}}$ be the minimal resolution of the associated canonical surface singularity.  The pair $(\RR, \widehat{\omega})$ specifies a $\bbC^\times$-action on $Z$, point-wise fixing the irreducible component of the exceptional locus of $Z$ corresponding to the marked fundamental weight $\widehat{\omega}$ under the McKay correspondence \cite{MR1886756}.
The associated Frobenius manifold is the $\bbC^\times$-equivariant quantum cohomology of $Z$,
    \[ \cM_{\rm GW} \coloneqq \mathrm{QH}_{\bbC^{\times}}(Z)\,.\]
    \item 
    In \cite{MR1727150,Dubrovin:1994hc}, Dubrovin constructs a Frobenius manifold structure $\mathrm{LG}(\lambda,\phi)$ on the Hurwitz moduli space of ramified covers of the projective line with given genus and ramification profile at infinity. Here, $\lambda$ (the \emph{Landau--Ginzburg superpotential}) denotes the universal map, and $\phi$ is the additional datum of a Saito  form \cite{MR723468} on the fibres of the family.  
    One can associate to the pair $(\RR, \widehat{\omega})$  an algebraically completely integrable system -- the type $(\RR,\widehat{\omega})$ affine relativistic Toda chain  \cite{Fock:2014ifa} at vanishing Casimir -- whose 
    isospectral dynamics 
    is encoded in a special Frobenius submanifold of a certain Hurwitz space \cite{Brini:2021pix}. In this context $\lambda$ is the spectral parameter of the relativistic Toda Lax matrix, and $\phi$ is the differential of the (logarithm of the) argument of its characteristic polynomial. We will denote this Frobenius manifold as
    \[
    \cM_{\rm LG} \coloneqq \mathrm{LG}(\lambda,\phi)\,.
    \]
    
\end{itemize}
\begin{figure}[t]
    \centering
\def\arraystretch{1.5}    
    \begin{tabular}{|c|c|c|}
    \hline
    $\RR$ & $\widehat{\omega}$ & Marked ADE Dynkin diagram \\
    \hline \hline 
      $\mathrm{A}_l$   & $\omega_{\bar k}$ &   \begin{tikzpicture}[scale=3] 
  \dynkin[label, root radius=0.06cm, labels={\omega_1,\omega_2,\omega_3,\omega_{\bar{k}}, \omega_{l-2},\omega_{l-1},\omega_l}] A{ooo.*.ooo}  
\end{tikzpicture} \\ \hline   

$\mathrm{D}_l$ &  $\omega_{l-2}$ &   \begin{tikzpicture}[scale=3]
  \dynkin[label,affine mark=*, root radius=0.06cm, labels={\omega_1~,\omega_2,\omega_3,\omega_j, \omega_{l-2},~\omega_{l-1},~\omega_l}] D{ooo.o.*oo}
\end{tikzpicture} \\ \hline 

$\mathrm{E}_6$ &   $\omega_3$   &   \begin{tikzpicture}[scale=3]
  \dynkin[label,affine mark=*, root radius=0.06cm, labels={~\omega_1,~\omega_6,\omega_2,\omega_3, \omega_{4},\omega_{5}}] E{ooo*oo}
\end{tikzpicture} \\ \hline 

$\mathrm{E}_7$
&   $\omega_3$ &      \begin{tikzpicture}[scale=3]
  \dynkin[label,affine mark=*, root radius=0.06cm, labels={\omega_1,~\omega_7,\omega_2,\omega_3, \omega_{4},\omega_{5},\omega_6}] E{ooo*ooo}
\end{tikzpicture} \\ \hline 

$\mathrm{E}_8$

&  $\omega_3$ &       \begin{tikzpicture}[scale=3]
  \dynkin[label,affine mark=*, root radius=0.06cm, labels={\omega_1,~\omega_8,\omega_2,\omega_3, \omega_{4},\omega_{5},\omega_6,\omega_7}] E{ooo*oooo}
\end{tikzpicture} \\ \hline 

    \end{tabular}

\caption{Marked Dynkin diagrams of pairs $(\RR, \widehat{\omega})$. The marked node corresponding to the weight $\widehat{\omega}$ is indicated in black.}
\label{fig:dynkin}
\end{figure}
\subsection{Dubrovin duality and mirror symmetry}
Given a Frobenius manifold $\cM$ with a linear Euler vector field $E$ and semi-simple product \[\circ : \Gamma(\cM, \cT\cM) \otimes_{\cO_\cM} \Gamma(\cM, \cT\cM)  \longrightarrow \Gamma(\cM, \cT\cM)\,,\]
one can construct \cite{MR2070050} a second family $\cM^\flat$ of Frobenius rings on the locus where $E$ is invertible in the $\circ$-algebra. This is obtained by pre-composing the flat pairing on $\cM$ with multiplication of each of its entries by $E^{-1/2}$:
\beq
\xymatrix{
    \cM \ar[r]^{E^{-1/2} \circ} & \cM^\flat
    }
\label{eq:adintro}
\eeq
In a canonical coordinate chart for $\cM$, \eqref{eq:adintro} corresponds to rescaling the coefficients of the (diagonal) Gram matrix of the Frobenius pairing by the inverse diagonal matrix of the canonical coordinates. It is further shown in \cite{MR2070050} that the resulting rescaled pairing on $\cM^\flat$ is flat, and induced by a \emph{dual} prepotential function to the original prepotential of $\cM$.   We will call $\cM^\flat$  the \emph{Dubrovin-dual} \footnote{The operation in \eqref{eq:adintro} is often referred to as an ``almost-duality'' of Frobenius manifolds, owing to the fact that $\cM^\flat$ is a weak (or {\it almost}) Frobenius manifold: we review this in \cref{sec:frobgen}. The terminology ``duality'' is perhaps somewhat improper, as for once the operation in \eqref{eq:adintro} is not involutive, but we abide by historical convention and refer to it as such (see also \cite{MR2183247,DavidStrachan}).}
Frobenius manifold of $\cM$. 

Dubrovin's duality acquires a particularly salient form when $\cM$ admits a realisation as a Landau--Ginzburg model on a family of algebraic curves.
In this case, \eqref{eq:adintro} amounts to replacing the superpotential by its logarithm 
\cite{Dubrovin:1994hc,MR2070050,MR2287835}:
\[\cM^\flat \simeq \mathrm{LG}(\log \lambda,\phi)\,.\]
%
From the point of view of the bihamiltonian quasi-linear integrable hierarchy defined on the loop space of a Frobenius manifold \cite{Dubrovin:1994hc}, a mirror-symmetry presentation of $\cM$ as a Landau--Ginzburg model is equivalent to a dispersionless Lax--Sato formulation of the integrable flows, with $\lambda$ coinciding with the Lax symbol \cite{MR1346289}. Dubrovin's duality \eqref{eq:adintro}, on the other hand, 
correponds to expressing the integrable flows on $\cM$ in Darboux coordinates for the second Poisson bracket of the Principal Hierarchy (the limiting value of the Dubrovin--Novikov pencil at infinity), and accordingly to a different (``dual'') notion of topological $\tau$-function in the sense of \cite{dubrovin:2001,MR3307147}.

\subsection{Main results}
In this paper we show that, for all marked ADE pairs $(\RR, \widehat{\omega})$, the Frobenius manifolds $\cM_{\rm AW}$, $\cM_{\rm GW}$ and $\cM_{\rm LG}$ are either non-trivially isomorphic, or Dubrovin-dual to each other. 

%
\begin{thm*}[=\cref{thm:msQH,thm:ddQH}]
For all marked ADE pairs $(\RR, \widehat{\omega})$, we have
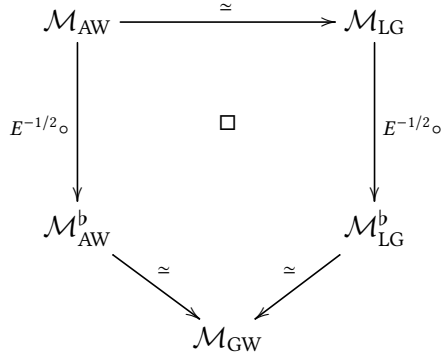
\[\xymatrix{
    \cM_{\rm AW} \ar[dd]_{E^{-1/2} \circ}  \ar[rr]^{\simeq} &   & 
    \cM_{\rm LG} \ar[dd]^{E^{-1/2} \circ} \\ & \Box & \\
   \cM^\flat_{\rm AW} \ar[dr]^{\simeq} &   & 
   \cM^\flat_{\rm LG}\ar[dl]_{\simeq}\\
   & \cM_{\rm GW} & \\
}\]
\end{thm*}
The isomorphism in the top row,
\[
\cM_{\rm AW} \simeq \cM_{\rm LG}\,,
\]
was proved in \cite{Brini:2021pix}. 
Therefore, the statement that the equivariant quantum cohomology of ADE resolutions is Dubrovin-dual to the corresponding extended affine Weyl Frobenius manifold,
\beq
\cM_{\rm GW} \simeq \cM_{\rm AW}^\flat\,, 
\label{eq:ddQHintro}
\eeq
is logically equivalent to proving its mirror realisation as an LG model on a family of relativistic Toda spectral curves, upon replacing the spectral parameter as $\lambda \longrightarrow \log \lambda$,
\beq
\cM_{\rm GW} \simeq \cM_{\rm LG}^\flat\,.
\label{eq:msQHintro}
\eeq
Our strategy will be to prove the mirror theorem \eqref{eq:msQHintro} first, as a means to establish the Dubrovin duality 
\eqref{eq:ddQHintro} as a consequence, for the classical series $\RR=\mathrm{A}_l$ and $\RR=\mathrm{D}_l$; and viceversa for the exceptional series $\RR=\mathrm{E}_l$.

\cref{thm:msQH} answers constructively, in all Dynkin types, a long-standing question about mirror symmetry in the fundamental setup of quantum cohomology of du Val resolutions. When $\RR=\mathrm{A}_l$, $Z$ is a smooth 2-dimensional toric Calabi--Yau  surface: in this case a Landau--Ginzburg mirror has been known since Givental's work on equivariant toric mirror symmetry \cite{MR1653024,MR2510741}.
The case where $\RR \neq \mathrm{A}_l$ and $Z$ is not toric has been outstanding to-date, as the methods of \cite{MR1653024} cannot be directly applied to this more general setup. An important consequence of 
\eqref{eq:msQHintro} is that it automatically provides a global integral representation of the components of the $J$-function as univariate Laplace-type integrals: for type $\RR = \mathrm{A}_l$, this enhanced control on their analytic continuation was brought to fruit in \cite{Brini:2013zsa} to prove  Iritani's integral K-theoretic and higher genus full-descendent Crepant Resolution Conjectures, using R-matrix quantisation techniques. \cref{thm:msQH} opens the way for a similar analysis for all ADE types, which will be explored in future work. The isomorphism \eqref{eq:ddQHintro} further suggests a conjectural integrable hierarchy governing the higher genus Gromov--Witten theory of $Z$: this should coincide with the one constructed in \cite{Milanov:2014pma} expressed in a suitable set of dual variables, given by Darboux co-ordinates for a second Hamiltonian structure. For the classical series $\RR=\mathrm{A}_l$ and $\RR=\mathrm{D}_l$, the hierarchy is a particular rational reduction of the 2-Toda hierarchy
\cite{Brini:2014mha}.


\subsection{Organisation of the paper}

The paper will be organised as follows. In \cref{sec:frobgen}, we will start with a short, but self-contained review of basic notions from the theory of Frobenius manifolds and Dubrovin's duality. We will follow this up in \cref{sec:frobRomega} with a detailed construction of the three Frobenius manifolds $\cM_{\rm AW}$, $\cM_{\rm LG}$ and $\cM_{\rm GW}$ specified by the datum of a marked ADE pair $(\RR, \widehat{\omega})$. In \cref{sec:LGsum}, we will explain how the embedding of the LG model $(\lambda, \phi)$ into a genus zero Hurwitz-Frobenius manifold for the classical series $\RR=\mathrm{A}_l$ and $\RR=\mathrm{D}_l$ allows to systematically determine the structure constants of $\cM_{\rm LG}^\flat$. Armed with this, the mirror theorem in \eqref{eq:msQHintro}, and therefore the duality \eqref{eq:ddQHintro} with the extended affine Weyl orbit spaces, can be deduced upon comparison with the genus zero Gromov--Witten calculations of \cite{MR2411404,Brini:2013zsa}. In \cref{sec:dd}, we give a general representation-theoretic argument,
applicable to all extended affine Weyl Frobenius manifolds, showing that the structure constants of the Frobenius product on both sides of \eqref{eq:ddQHintro} belong to a certain finite-dimensional vector space of quasi-homogeneous polynomials. As such, the corresponding $(2,1)$-tensors are  determined by their values on a (small) finite set of points $\mathfrak{I}$ in the semi-simple locus of $\cM_{\rm AW}$, which we call a set of \emph{initial conditions} for $\cM_{\rm AW}$. We perform this analysis specifically for $\RR=\mathrm{E}_l$, and show how the reduction to initial conditions drastically simplifies the verification of \eqref{eq:ddQHintro}, and therefore the proof of the mirror theorem in \eqref{eq:msQHintro}, which we carry out for the entire exceptional series. 

 We shall never assume Einstein's convention in this paper.  For the reader's convenience, we collect the notation employed throughout the text in \cref{tab:notation}.

\begin{table}[t]
\begin{center}
\begin{tabular}{ |c|l| } 
 \hline
 $\mathcal{R}$, resp. $\mathcal{R}^+$ & {An ADE root system, resp. its set of positive roots}\\
  \hline
 $\Pi = \{\alpha_1,\dots,\alpha_l\}$ & {The set of simple roots of $\mathcal{R}$}\\
 \hline
 $\Omega=\{\omega_1,\dots,\omega_l\}$ & {The set of fundamental weights of $\mathcal{R}$}\\
  \hline
   $\rho_i$, resp. $\Gamma_i$ & {Irreducible representation, resp. weight system, with highest weight $\omega_i$}\\
  \hline
 $\left(\mathfrak{h},\langle,\rangle\right)$ & {The Cartan subalgebra of $\mathfrak{g}$, together with its Cartan--Killing form}  \\ 
\hline
$\mathfrak{g}$ & {The complex simple Lie algebra associated to $\RR$} \\ 
  \hline
 $C_{ij}=\langle\alpha_i,\alpha_j\rangle$ & {The Cartan matrix of $\mathcal{R}$}\\
 \hline
 $\widehat{\omega}$ & A marked fundamental weight of $\RR$ as in \cref{fig:dynkin}.\\
 \hline
 $\mathcal{W}/\widehat{\mathcal{W}}/\widetilde{\mathcal{W}}$ & The Weyl/affine Weyl/extended affine Weyl group of type $(\RR, \widehat{\omega})$\\
 \hline
 $X$, resp. $Z$ & The du Val singularity of type $\mathcal{R}$, resp. its minimal resolution\\
 \hline
 $\mathsf{T}$ & The torus action on $Z$ specified by $\widehat{\omega}$
 \\ \hline
 $\cM_{\rm AW}$ & The extended affine Weyl Frobenius manifold of type $(\RR,\widehat{\omega})$\\
 \hline
 $\cM_{\rm GW}$ & The equivariant quantum cohomology of $Z$ of type
 $(\RR, \widehat{\omega})$
 \\
  \hline
 $\cM_{\rm LG}$ & The Landau--Ginzburg Frobenius manifold of the affine relativistic  \\
  & Toda chain of type $(\RR, \widehat{\omega})$
\\
  \hline
 $(\eta_{\bullet}, c_\bullet, \eta^\flat_\bullet, c_\bullet^\flat)$
 & The metric, product tensor, and their duals on $\cM_\bullet$ ($\bullet\in \rm \{AW, GW, LG\}$) 
 \\
 \hline
 $(Y_\alpha)_{\alpha=1}^l$ & Basic Weyl-invariant Fourier polynomials of $\mathfrak{h}$\\
 \hline
 $(W_\alpha)_{\alpha=1}^l$ & The fundamental characters of $\exp(\mathfrak{g})$ \\
  \hline
 $(y_\alpha)_{\alpha=1}^{l+1}$ & Basic extended affine Weyl-invariant Fourier polynomials of $\mathfrak{h} \oplus \bbC$\\
 \hline
 $(x_i)_{i=1}^{l+1}$, resp. $(p_i)_{i=1}^{l+1}$ & Linear coordinates for $\mathfrak{h} \oplus \bbC$ in the basis $\Pi^\vee$, resp. $\Omega^\vee$
 \\
   \hline
 $(t_A)_{A=1}^{l+1}$ & A flat coordinate chart for $\cM_{\rm AW}$ 
 \\
 \hline
     $\mathfrak{I}$ & A set of initial conditions for $\cM_{\rm AW}$ \\
 \hline
\end{tabular}
\end{center}
\medskip
\centering
\caption{Notation used throughout the text. When working in local coordinates for $\cM_{\rm AW}$, we will consistently use lower-case Latin indices for a flat chart for the intersection form; upper-case Latin indices for a flat chart for the Saito metric; and Greek indices for a coordinate chart given by basic invariants.}
\label{tab:notation}
\end{table}



\subsection*{Acknowledgements} We thank J.~Bryan, H.~Iritani and H-H.~Tseng for correspondence. A.~B. was supported by the EPSRC  Early Career Fellowship EP/S003657/2. J.~M. was supported by a PhD studentship of the  EPSRC Doctoral Training Partnership EP/W524360/1.

\section{Setup}
\subsection{Generalities on Frobenius manifolds}
\label{sec:frobgen}

Let $M$ be an $n$-dimensional complex manifold. We will write $\cT_M$ (resp. $\Omega_M$) for the sheaf of holomorphic sections of the holomorphic tangent (resp. cotangent) bundle $T^{1,0}M$ (resp. $T^*_{1,0}M$), and $\mf{X}(M) \coloneqq \mathrm{H}^0(M, \cT_M)$ for the space of global holomorphic vector fields on $M$. A holomorphic Frobenius manifold structure\footnote{We call \emph{Frobenius manifold} here what is often elsewhere referred to as a \emph{weak} (or \emph{almost}) \emph{Frobenius manifold}, since we do not require axiom {\bf FM6} (covariant linearity of $E$) to hold in our definition. The classical notion of Frobenius manifold in \cite{Dubrovin:1994hc} is what we call \emph{conformal Frobenius manifold} in this paper. We will reserve the terminology ``almost Frobenius manifold'' to indicate a Frobenius manifold where 
{\bf FM4} (covariant constancy of $e$) 
is dropped.} on $M$ is a 4-tuple $\cM\coloneqq =(M,  c, \eta , e)$ satisfying the following axioms:
\begin{description}
\item[FM1] the {\it metric} $\eta \in \mathrm{H}^0(M,\mathrm{Sym}^2 \Omega_M)$ is a flat perfect symmetric pairing on $\cT_M\,$;
\item[FM2] the {\it product} $c \in \mathrm{H}^0(M,\mathrm{Sym}^3 \Omega_M)$ is a totally symmetric $(0,3)$-tensor 
\[
\eta(X, Y \circ Z) = \eta(X \circ Y, Z) \coloneqq c(X, Y, Z) \in \cO_M(M) \,, \quad X, Y, Z \in \mathfrak{X}(M)\,,
\]
inducing, holomorphically in $p\in M$, a structure of commutative, unital, associative Frobenius algebra on the tangent fibre $T^{1,0}_pM$\,;
\item[FM3] $\nabla_{W}c(X, Y, Z)$ is totally symmetric in $W, X, Y, Z \in \mf{X}(M)$, where $\nabla$ denotes the Levi-Civita connection of $\eta$;
\item[FM4] the {\it identity vector field} $e \in \mf{X}(M)$, defined by its fibrewise restriction to the identity of the algebra, is horizontal with respect to $\nabla$, $\nabla e = 0\,.$
\end{description}
Two supplementary conditions are often imposed on $\cM$, the second of which may or may not be realised in the context of this paper. 
\begin{description}
\item[FM5] $\cM$ is {\it semi-simple} if the set \[\mathrm{Discr}(\cM) \coloneqq \{p\in M~|~\exists\, v \in T_p M~{\rm with}~ v \circ v=0\}\] has positive complex co-dimension;
\item[FM6] $\cM$ is {\it conformal} if there exists 
a holomorphic vector field $E \in \mathfrak{X}(M)$ which is covariantly linear, $\nabla\nabla E=0$, and is such that
\[
\mathcal{L}_{ E} \circ  =  \circ \,, \quad
\mathcal{L}_{ E} \eta  = (d-2) \eta\,,
\]
for some constant $d\in \bbQ$, known as the {\it charge} of $\cM$. 
\end{description}
All examples in this paper will be semi-simple, but not necessarily conformal. We will write \[M^{\rm ss} \coloneqq M\setminus \mathrm{Discr}(\cM)\] to indicate the open semi-simple locus of $\cM$, and the calligraphic notation $\cM^{\rm ss}$ for the Frobenius manifold structure induced by restriction of $\cM$ to $M^{\rm ss}$.

Since the metric $\eta$ is flat, the $\mathrm{GL}(n,\bbC)$-equivalence class of flat frames for $\eta$ equips a Frobenius manifold with a canonical $\mathrm{GL}(n,\bbC)\ltimes \bbC^n$ affine-linear equivalence class of flat charts, in which the Gram matrix of the pairing $\eta$ is constant. The axioms {\bf FM1-FM6} in one such chart $(t_1, \dots , t_{n})$ amount to the local existence of a holomorphic function $F$ (the \emph{prepotential}) with the following properties:

\ben[(i)]
\item the unit vector field is \[ e = \frac{\de}{\de{t_1}}\,;\]
    \item the Gram matrix $\eta_{AB} \coloneqq \eta\l(\frac{\de}{\partial{t_A}}, \frac{\de}{\partial t_B}\r)$ is constant, non-degenerate, and equal to the Hessian matrix of $\de_{t_1} F$, \[\eta_{AB}= \frac{\partial^3 F}{\de {t_1}  \de t_A \de t_B} \,;\]
    \item the prepotential is weighted quasi-homogeneous in its arguments, \[\sum_{A=1}^n \l(p_A t_A \frac{\de}{\partial t_A}  + r_{A} \frac{\de}{\partial t_A}\r) F = (3-d) F\,,\] 
    for some $p_A$, $q_A$, $d \in \bbQ$; 
    \item for all $A,B,M,N \in \{1,\dots,n\}$, writing \[\eta^{CD} \coloneqq (\eta^{-1})_{CD}\,,\] the prepotential satisfies the Witten--Dijkgraaf--Verlinde--Verlinde (WDVV) equations,
\begin{equation}
\sum_{C,D=1}^n    \l(\dfrac{\partial^3 F}{\partial t_A \partial t_B \partial t_C} \eta^{CD}\dfrac{\partial^3 F}{\partial t_D \partial t_M \partial t_N}-\dfrac{\partial^3 F}{\partial t_A \partial t_M \partial t_C} \eta^{CD}\dfrac{\partial^3 F}{\partial t_D \partial t_B \partial t_N}\r)  = 0\,. 
\label{eq:wdvv}
\end{equation}
\een

\subsubsection{Dubrovin duality}

If $\cM$ is semi-simple, and for all $p\in M$ such that $E|_p$ is in the group of units of the $\bbC$-algebra $(T^{1,0}_p M, \circ)$,  one may define a second metric $\eta^\flat$ (the {\it intersection form}) by 
\begin{equation}\label{eq:dualmetric}
    \eta^\flat\l(X, Y\r) \coloneqq \eta\l(E^{-1/2} \circ X, E^{-1/2} \circ Y\r) = \eta\l(E^{-1} \circ X, Y\r). 
\end{equation}
A consequence of {\bf FM5-FM6} \cite[Lect.~3]{Dubrovin:1994hc} is that 
the metric
$\eta^\flat$ is also flat;
we will usually denote by $(x_1, \dots, x_n)$ a choice of a flat coordinate chart for it. Alongside $\eta^\flat$, we can define a second commutative, associative, $\cO_{M^{\rm ss}}$-algebra structure on $\cX(M^{\rm ss})$ with unit $E$ via 
\beq
c^\flat (X, Y, Z) \coloneqq \eta^\flat(X, Y \star Z)\,, \quad 
    X\star Y  \coloneqq  E^{-1}\circ X \circ Y\,. 
    \label{eq:dualprod}
\eeq
The $\star$-product is compatible (in the sense of Frobenius algebras) with the pairing $\eta^\flat$: this is an immediate consequence of the $\circ$-product being compatible with respect to the pairing $\eta$. 

\begin{prop}[\cite{MR2070050,MR2836400}]
Let $\cM=(M,c,\eta, e,E)$ be a semi-simple conformal Frobenius manifold of charge $d=1$. Then $\cM=(M^{\rm ss}, c^\flat, \eta^\flat,E)$
is a semi-simple Frobenius manifold satisfying axioms {\bf FM1-4}, with flat unit $E \in \cX(M^{\rm ss})$.
\label{prop:dd}
\end{prop}

In particular, if $d=1$ and in a flat chart $(x_1, \dots, x_n)$ for the intersection form $\eta^\flat$, there exists a solution $F^\flat$ of the WDVV equations \eqref{eq:wdvv} (with $\eta$, $(t_A)_\alpha$ replaced by $\eta^\flat$, $(x_i)_i$) where furthermore 
\[
E= \frac{\de}{\de x_1}\,, \qquad \eta^\flat_{ij} = \frac{\de^3 F^\flat}{\de x_1 \de x_i \de x_j}\,.
\]

The Frobenius manifold structure $\cM^\flat = (M^{\rm ss}, c^\flat, \eta^\flat, E)$ induced by $F^\flat$ will be called the {\it Dubrovin-dual} structure to $\cM$. By \cref{prop:dd} and \eqref{eq:dualprod}, it will satisfy axioms {\bf FM1-5}, but it will not generally satisfy the additional conformality axiom {\bf FM6}.
The two perfect pairings $\eta$ and $\eta^\flat$ on $\cT_M^{\rm ss}$ locally induce isometries between tangent and cotangent fibres,
\beq
\l(T_p M^{\rm ss}, \eta\r) \simeq \l(T^*_p M^{\rm ss}, \eta^{-1}\r)\,, \qquad 
\l(T_p M^{\rm ss}, \eta^\flat\r) \simeq \l(T^*_p M^{\rm ss}, (\eta^\flat)^{-1}\r)\,.
\label{eq:isocotg}
\eeq
In Einstein's convention, this would be the familiar operation of ``raising the indices'' with $\eta^{AB}$ (resp. $(\eta^\flat)^{ij}$). The isomorphisms \eqref{eq:isocotg} define two, a priori distinct, $\cO_{M^{\rm ss}}$-algebra structures on holomorphic 1-forms  $\theta, \chi \in \Omega_{M^{\rm ss}}$,
\beq
\theta\, \widehat{\circ}\, \chi\,, \qquad \theta\, \widehat{\star}\, \chi\,, 
\eeq
respectively given by the dual of the $\circ$-product under $\eta^{-1}$, and the dual of the $\star$-product under $(\eta^\flat)^{-1}$ in \eqref{eq:isocotg}. By \eqref{eq:dualmetric} and \eqref{eq:dualprod}, these two dual products on $\Omega_M$ are in fact identically isomorphic:
\beq
\theta\, \widehat{\circ}\, \chi = \theta\, \widehat{\star} \chi\,.
\label{eq:ddprod_id}
\eeq
Spelling out \eqref{eq:ddprod_id} in the respective flat charts for $\eta$ and $\eta^\flat$ results in the following non-trivial relation between the prepotentials of $\cM$ and $\cM^\flat$:
%
\begin{equation}
    \frac{\partial^3 F^\flat}{\partial x_{i}\partial x_{j}\partial x_{k}} = \sum_{a,b,A,B,C,M,N=1}^n\eta^\flat_{ia} \eta^\flat_{jb}\frac{\partial t_C}{\partial x_k}\frac{\partial x_a}{\partial t_A}\frac{\partial x_b}{\partial t_B} \eta^{AM} \eta^{BN}
    \frac{\partial^3 F}{\partial t_C \partial t_M\partial t_N}\,.
    \label{eq:AD_formula}
\end{equation}

\subsection{Frobenius manifolds of type $(\RR, \widehat{\omega})$}
\label{sec:frobRomega}

With these preliminaries, we shall consider three classes of semi-simple Frobenius manifolds associated to a marked ADE pair $(\RR, \widehat{\omega})$.

\subsubsection{Extended affine Weyl Frobenius manifolds}

Let $\mathfrak{g}=\mathfrak{ade}_l$ be the rank-$l$ complex, simple, simply-laced Lie algebra with root system $\mathcal{R}$. We shall denote by $\mathfrak{h}$ the associated Cartan subalgebra,  and by $\mathcal{W}$ its Weyl group. The action of $\mathcal{W}$ on $\mathfrak{h}$ lifts to an action of the affine Weyl group $\widehat{\mathcal{W}} \cong \mathcal{W} \ltimes  \Lambda^\vee_r$, with  $\Lambda_r^\vee$  the lattice of co-roots:
 \begin{align}
    \widehat{\mathcal{W}} \times \mathfrak{h} & \longrightarrow \mathfrak{h}, \nn \\
    ((w, \alpha^\vee), h) & \longrightarrow w(h) + 2\pi\ri\alpha^\vee.
\end{align}
For $(\RR, \widehat{\omega})$ a marked ADE pair, the corresponding $\widehat{\omega}$-\textit{extended} affine Weyl group $\widetilde{\mathcal{W}}$
 is defined as the semi-direct product $\widetilde{\mathcal{W}} \coloneqq \widehat{\mathcal{W}} \ltimes \mathbb{Z}$ acting on $\mathfrak{h} \oplus \mathbb{C}$ by 
\begin{align}
    \widetilde{\mathcal{W}} \times {\mathfrak{h}} \oplus \mathbb{C} & \longrightarrow \mathfrak{h} \oplus \mathbb{C}, \nn \\
    ((w, \alpha^\vee, {l}), (h, v)) & \longrightarrow (w(h) + 2\pi\ri \alpha^\vee + 2\pi\ri l \widehat{\omega}, {v} - 2\pi\ri l){.} 
    \label{eq:extweyl}
\end{align}
Let  $\Sigma$ denote the hyperplane arrangement associated to the root system $\RR$, and $\mathfrak{h}^{\text{reg}} \coloneqq \mathfrak{h} \backslash  \Sigma $ be the set of regular elements in $\mathfrak{h}$. The restriction of \eqref{eq:extweyl} to $\mathfrak{h}^{\rm reg} \oplus \bbC$ is a free affine action, whose quotient defines the regular orbit space of the extended affine Weyl group of $\RR$ with marked weight $\widehat{\omega}$ as
\begin{equation}
M_{\rm AW} \coloneqq (\mathfrak{h}^{\text{reg}} \times \mathbb{C})/ \widetilde{\mathcal{W}}  
\cong \mathcal{T}^{\text{reg}}/\cW \times \mathbb{C}^*, 
   \label{eq:EAWmaniLG}
\end{equation}
where $\mathcal{T}^{\text{reg}} = \exp( \mathfrak{h}^{\text{reg}})$ is the image of the set of regular elements of $\mathfrak{h}^{\rm reg}$ under the exponential map to the maximal torus $\mathcal{T}$. 
Let $(x_1, \dots , x_{l})$ be linear coordinates on $\mathfrak{h}$ w.r.t. the co-root basis $\{\alpha_1^\vee, \dots , \alpha_{l}^\vee\}$, and extend these to linear coordinates $(x_1, \dots , x_{l}; x_{l+1})$ on $\mathfrak{h} \oplus \bbC$, giving local coordinates on the regular orbit space. Denoting
$\bra \alpha, \beta \ket$ the pairing on $\mathfrak{h}^*$ induced by the restriction of the Killing form on the Cartan subalgebra, and writing \[C_{ab} = \left\langle \alpha_a,  \alpha_b^\vee\right\rangle\,, \quad d_a \coloneqq \bra \omega_a, \widehat{\omega}\ket\,, \quad \widehat{d} \coloneqq \bra \widehat{\omega}, \widehat{\omega}\ket\,,\] 
we can define a non-degenerate pairing $\xi$ on  $\mathfrak{h} \times \mathbb{C}$ by orthogonal extension of minus the Cartan--Killing form on $\mathfrak{h}$ as
\begin{equation}
      \eta^\flat_{\rm AW}(\partial_{x_a}, \partial_{x_b}) \coloneqq
  \begin{cases}
                                   -C_{ab} & \text{if } a, b< l+1, \\
                                   \widehat{d} & \text{if } a=b=l+1, \\
                                   0 & \text{otherwise,}
  \end{cases}
\label{eq:inter_form}
\end{equation}
with $x_{l+1}$ parametrising linearly the right summand in $\mathfrak{h} \oplus \mathbb{C}$. 
The quotient map $\aleph: \mathcal{T}^{\text{reg}} \times \mathbb{C}^* \rightarrow M_{\rm AW}$ from \eqref{eq:EAWmaniLG} defines a principal {$\mathcal{W}$}-bundle on $M_{\rm AW}$: a section $\tilde{\sigma}_i$ lifts a (sufficiently small) open $U \subset M_{\rm AW}$ to the $i^\text{th}$ sheet of the cover $V_i \in \widetilde{\sigma}_i^{-1}(U) \equiv V_1 \sqcup \dots \sqcup V_{|\mathcal{W}|} $. The invariant ring $\cI \coloneqq \bbC[\cT]^{\cW}$ is, from classical results about exponential Weyl invariants \cite{B02}, a polynomial ring
$\cI \simeq \bbC[Y_1, \dots , Y_l],$
where 
    \beq 
    Y_i \coloneqq S_\cW(\re^{\bra \omega_i, h\ket})\,,
    \label{eq:Yi}
    \eeq
and $S_\cW$ is the average over the Weyl orbit. Equivalently, we have 
$\cI \simeq \bbC[W_1, \dots , W_l],$
where
\beq 
W_i \coloneqq \sum_{\omega\in \Gamma_i}\re^{\bra \omega, h\ket}\,.
\label{eq:Wi}
\eeq

\begin{prop}[\text{\cite[Thm.~1.1]{MR1606165}}]
Let  $\mathcal{A}$ be the ring of $\widetilde{\cW}$-invariant Fourier polynomials in the variables $x_1,\dots, x_l$ and $\frac{x_{l+1}}{f}$ that are bounded in the limit:
    \begin{equation}
        x = x^{[0]} + \widehat{\omega} \tau ,\quad x_{l+1} = x^{[0]}_{l+1} - \tau ,\quad \tau \to +\infty,
    \label{limit}
    \end{equation}
    for any $\big(x^{[0]},x^{[0]}_{l+1}\big)$, where $f$ is the determinant of the Cartan matrix. Then 
    \[
    \mathcal{A} \simeq \bbC\l[y_1, \dots , y_l, \re^{y_{l+1}}\r]\,,
    \]
with
\beq
y_\alpha= 
\begin{cases}
\re^{ d_\alpha x_{l+1}}Y_\alpha\,, & \alpha=1,\dots, l\,, \\
x_{l+1}\,, & \alpha=l+1\,.
\end{cases}
\label{eq:yi}
\eeq
and $d_i := \langle \omega_i,\widehat{\omega} \rangle$. 
\label{prop:Aring}
\end{prop}
In the following, for $k \in \{1,\dots, l\}$ such that $\omega_{\bar k} = \widehat{\omega}$, we will denote
$\widehat{y}\coloneqq y_{\bar k}$.
We will define a grading on $\cA$ by 
\begin{align*}
\deg\, y_\alpha =  d_\alpha \, (\alpha=1,\dots, l)\,, \qquad \deg\, \re^{y_{l+1}}= 1\,.
\end{align*}
The following reconstruction theorem holds \cite[Thm~2.1]{MR1606165}.

\begin{thm}
There exists a unique, up to isomorphism, semi-simple and conformal Frobenius manifold \[\cM_{\rm AW}=(M_{\rm AW}, c_{\rm AW},  \eta_{\rm AW}, e, E)\] of charge $d=1$ satisfying the following properties in flat coordinates $(t_1, \dots , t_{l+1})$ for $\eta_{\rm AW}$:
\begin{description}
    \item[AW-I] $\widehat{d} E =  \partial_{x_{l+1}} = \sum_{J=1}^{l} d_J t_J \partial_{{t}_J} + \partial_{{t}_{l+1}}$;
    \item[AW-II] the intersection form is $\eta^\flat = \tilde{\sigma}_i^* \xi$;
    \item[AW-III] the prepotential is polynomial in $t_1, \dots, t_{l+1}$, $\re^{t_{l+1}}$.
\end{description}
\label{thm:EAW}
\end{thm}

\subsubsection{Quantum cohomology of ADE resolutions}
\label{sec: equi_geometry}

Let $\mathsf{G}<\mathrm{SU}(2)$, $|\mathsf{G}|<\infty$ be a finite subgroup of $\mathrm{SU}(2)$. The classical McKay correspondence  classifies $\mathsf{G}$ by root systems of type $\RR=\mathrm{ADE}_l$:
\[
\mathsf{G} \simeq 
\begin{cases}
    \bbZ/(l+1)\bbZ & \RR=\mathrm{A}_l\,, \\
    BD_{4l-2}=\bra 2,2,l\ket  & \RR=\mathrm{D}_l\,, \\
    BT=\bra 2,3,3\ket & \RR=\mathrm{E}_6\,, \\
    BO=\bra 2,3,4\ket & \RR=\mathrm{E}_7\,, \\
    BI=\bra 2,3,5\ket & \RR=\mathrm{E}_8\,. \\
\end{cases}
\]

The corresponding du~Val singularity is defined as the affine scheme \[ X \coloneqq \mathrm{Spec}\l(\mathbb{C}[x,y]^{\mathsf{G}}\r)\,,\]
where the action on $\mathbb{C}^2$ by $\mathsf{G}$ is induced by restriction of the fundamental representation of $SU(2)$. Klein's classical presentation of the ring of invariants, 
\[
\mathbb{C}[x,y]^{\mathsf{G}} \simeq \frac{\mathbb{C}[u,v,w]}{\bra \cI_\mathsf{G}\ket }\,,
\]
where
\[
\cI_\mathsf{G}(u,v,w) = 
\begin{cases}
    w^2+u^2+v^{l+1} & \RR=\mathrm{A}_l\,, \\
    w^2+v(u^2+v^{l-2})&  \RR=\mathrm{D}_l\,, \\
    w^2+u^3 +v^4 & \RR=\mathrm{E}_6\,, \\
    w^2+u^3 +u v^3 & \RR=\mathrm{E}_7\,, \\
    w^2+u^3 +v^5 & \RR=\mathrm{E}_8\,, \\
\end{cases}
\]
realises $X$ as a hypersurface in $\bbC^3$ with an isolated singularity at the origin. There is a well-known canonical minimal resolution 
\beq
\begin{tikzcd}
Z  \arrow[r, "\pi"] &  X
\end{tikzcd}
\label{eq:piZX}
\eeq
obtained through a sequence of blowing-ups of the singularity \cite{MR1886756}. The intersection diagram describing the configuration of irreducible rational curves in the exceptional locus is the Dynkin diagram of the corresponding ADE type, and the integral homology of the resolution, \[\mathrm{H}_\bullet(Z, \bbZ) = \mathrm{H}_0(Z, \bbZ)\oplus \mathrm{H}_2(Z, \bbZ) \simeq \bbZ \oplus \bbZ^l\,,\]
is isomorphic to the affine root lattice of type $\RR$. Writing
\[\Pi \coloneqq \{\alpha_1, \dots, \alpha_l\}\,, \qquad \Omega \coloneqq \{\omega_1, \dots, \omega_l\}
\]
for, respectively,  the set of simple roots and fundamental weights of the complex simple Lie algebra associated to $\RR$. 
we will label the irreducible components $\{e_1,\dots,e_l\}$ of the exceptional locus accordingly, so that 
\beq 
e_i \longleftrightarrow \alpha_i \longleftrightarrow \omega_i\,.
\label{eq:mckaydivroots}
\eeq

Let $\mathsf{T}\simeq \bbC^\times$ and consider a $\mathsf{T}$-representation on $\bbC^2$  commuting with the action of $\mathsf{G}$. This induces $\mathsf{T}$-actions on $X$ and $Z$, respectively by descent to the quotient and by $\mathsf{T}$-equivariance of the resolution \eqref{eq:piZX}. 
When $\RR=\mathrm{DE}_l$, we have a unique possible choice for the action of $\mathsf{T}$: this is the scalar torus action on $\bbC^2$ with characters $(t,t)$ on the affine coordinates $(x,y)$ of $\bbC^2$. 
When $\RR=\mathrm{A}_l$, since $\mathsf{G}$ is abelian, we have more generally that the full Cartan torus \[ \mathsf{T}'\coloneqq (\bbC^\times)^2 < \mathrm{GL}(2,\bbC)\,,\] with characters $(t_x, t_y)$ 
will act effectively on $X$ and $Z$. For $\bar k \in \{1,\dots,l\}$, we will be specially interested in the one-dimensional subtori $\mathsf{T}$ acting with characters $(t_x,t_y)=(t^{l+1 -\bar k}, t^{\bar k})$.

As the singularity is the only torus fixed point in $X$, the $\mathsf{T}$-fixed locus of $Z$ is fully contained in the exceptional set. We will write $[e_i]^\vee$ for the Poincar\'e dual classes to $[e_i]$ in the locally compact cohomology, $\varphi_i$ for their canonical lifts to $\mathsf{T}$-equivariant cohomology, and $\varphi_{l+1}$ for the identity class in $\mathrm{H}_{\mathsf{T}}(Z)$. The $\mathsf{T}$-equivariant cohomology of $Z$ is an $(l+1)$-dimensional vector space over the field of fractions of $H(B\mathsf{T}, \bbC) \simeq \bbC[\nu]$, where $\nu$ is the first Chern class of the $\mathsf{T}$-representation $t$.  For $\RR=\mathrm{A}_l$, we can more generally consider the equivariant cohomology of $Z$ with respect to the full 2-dimensional torus $\mathsf{T}'$: this will now be an $(l+1)$-dimensional vector space over $\bbC(\nu_1, \nu_2)$, with $\nu_i \coloneqq c_1(t_i)\in \mathrm{H}_{\mathsf{T}'}(\mathrm{pt})$. We will slightly abuse notation and continute to write $\varphi_1,\dots,\varphi_l,\varphi_{l+1}$ for the lift of $e^\vee_i$ in the $\mathsf{T}'$-equivariant cohomology of $Z$.

\begin{rmk}
For $(\RR, \widehat{\omega})$ a marked ADE pair, the marked fundamental weight
$\widehat{\omega} \in \Omega$ corresponds to the (single) irreducible exceptional divisor $\widehat{e} \in \mathrm{H}_2(Z, \bbZ)$ which is point-wise fixed by $\mathsf{T}$
\[
\widehat{e} \longleftrightarrow \widehat{\alpha} \longleftrightarrow \widehat{\omega}\,.
\]
To see this, notice that each double point in the exceptional locus is $\mathsf{T}$-fixed. For $\RR=\mathrm{DE}_l$, 
this entails that the $\mathsf{T}$-action on the irreducible component $\widehat{e}$ labelled by the trivalent marked node in \cref{fig:dynkin} must be trivial, since $\widehat{e}$ is a $\bbP^1$ with three $\mathsf{T}\simeq \bbC^\times$-fixed points. 
When $\RR=\mathrm{A}_l$, the 
weights $w_i^{\pm}$ at the two $\mathsf{T}'$-fixed points of $e_i$ are \cite[App.~B]{Brini:2013zsa} 
\[
(w_i^{+}, w_{i}^{-}) = (- i \nu_1 + (l+1 - i) \nu_2,  i \nu_1 - (l+1 - i)\nu_2)\,.
\]
As discussed above, the restriction to a one-dimensional torus $\mathsf{T}$ acting with characters $(t^{l+1-\bar k}, t^{\bar k})$ sets
\beq
\nu_1 = (l+1-\bar k) \nu\,, \qquad 
\nu_2 = \bar k \nu\,,
\label{eq:ZxCtoZ}
\eeq
so that $w_{\bar k}^\pm =0.$ Hence, in this case,
\[
e_{\bar k} \longleftrightarrow \alpha_{\bar k} \longleftrightarrow \omega_{\bar k}=\widehat{\omega}\,.
\] 
\label{rmk: 1d_torus}
\end{rmk}
\vspace{-.5cm}
The torus action on $Z$ induces an 
action on its moduli space of stable maps $\overline{\mathcal{M}}_{g,n}(Z,\beta)$. Since the $\mathsf{T}$-action on  $Z$ is free away from the exceptional locus, its fixed locus in $\overline{\mathcal{M}}_{g,n}(Z,\beta)$ is proper, and it carries a torus-equivariant perfect obstruction theory and  virtual fundamental class \cite{MR1666787}. Gromov--Witten invariants of $Z$ can then be defined by 
\[
\langle \varphi_{i_1}, \dots, \varphi_{i_n}\rangle_{g,n,d}^{Z}:=\intEquiv{\mathsf{T}}_{[\overline{\mathcal{M}}_{g,n}(Z,
\beta)]_{\mathsf{T}}^{\rm vir}} \prod_{j=1}^n \mathrm{ev}_j^*(\varphi_{i_j})
\coloneqq
\int_{[\overline{\mathcal{M}}^\mathsf{T}_{g,n}(Z,
\beta)]_{\mathsf{T}}^{\rm vir}} \frac{\prod_{j=1}^n \iota^*\mathrm{ev}_j^*(\varphi_{i_j})}{e_\mathsf{T}(N^{\rm vir})} \in \bbQ(\nu)\,,
\]
with $\iota:  \overline{\mathcal{M}}^\mathsf{T}_{g,n}(Z,
\beta) \hookrightarrow \overline{\mathcal{M}}_{g,n}(Z,
\beta)$ the immersion of the substack of $\mathsf{T}$-fixed points into the moduli space of stable maps, and $N^{\rm vir}$ its $\mathsf{T}$-equivariant normal bundle. 
Writing $x=\sum_{i=1}^{l+1} x_i \varphi_i \in \mathrm{H}_\mathsf{T}(Z)$, 
the genus-zero invariants of $Z$ define a Frobenius manifold satisfying the axioms {\bf FM1-FM5}:
\[\cM_{\rm GW}\coloneqq(\mathrm{H}_\mathsf{T}(Z), c_{\rm GW}, \eta_{\rm GW}, \varphi_{l+1})\,\]
via
\beq
c_{\rm GW}(\varphi, \psi, \theta) 
 \coloneqq  \intEquiv{\mathsf{T}}_{[Z]_{\mathsf{T}}}  \varphi \cup \psi \cup \theta + \sum_{\beta\in \mathrm{H}_2(Z, \bbZ)} \langle \varphi,\psi,\theta \rangle_{0,3,\beta}^{Z}\re^{\langle\beta, x\rangle}
 \,, 
\label{eq:cGW}
\eeq
\[
\eta_{\rm GW}(\varphi, \psi)  \coloneqq  c_{\rm GW}(\varphi_{l+1}, \varphi, \psi) = \intEquiv{\mathsf{T}}_{[Z]_{\mathsf{T}}} \varphi \cup \psi\nn\,.
\]
In \cite{MR1606165}, when the torus acts with characters $(t,t)$, the authors solved the genus zero $\mathsf{T}$-equivariant Gromov--Witten theory of $Z$ using degeneration arguments and the Aspinwall--Morrison formula for the super-rigid local curve.
The formal power series \eqref{eq:cGW} is the Fourier expansion of a trigonometric rational function given explicitly by
\[
c_{\rm GW}(\varphi_{l+1},\varphi_i,\varphi_j) = 
\,\eta_{\rm GW}(\varphi_i, \varphi_j)=
\begin{cases}
\frac{1}{\nu^2}\frac{1}{|\mathsf{G}|}\,, & i=j=l+1\,, \\
-C_{ij}\,, & i,j\leq l \,, \\
0\,, & \mathrm{else}\,,
\end{cases}
\]
\beq
c_{\rm GW}(\varphi_i,\varphi_j,\varphi_k) =-\nu\sum_{\beta\in \mathcal{R}^+}
\langle\alpha_i,\beta\rangle \langle\alpha_j,\beta\rangle \langle\alpha_k,\beta\rangle
\coth\l(\frac{\sum_{m=1}^l \langle\beta,\alpha_m\rangle x_m}{2} \r)\,, \quad i,j,k\leq l\,.
\label{eq:cBG}
\eeq
%
The corresponding prepotential is $F_{\rm GW}=F^0_{\rm GW}+F^+_{\rm GW}$, where the zero and positive degrees parts read
\begin{align}
    \label{GW_potential}
    F^0_{\rm GW}= & \frac{1}{6 \nu^2}\frac{1}{|\mathsf{G}|}(x_{l+1})^3 -\frac{1}{2}x_{l+1}\sum_{i,j=1}^l C_{ij} x_i x_j
    -\frac{\nu}{6}\sum_{\beta\in \mathcal{R}^+}\sum_{i,j,k=1}^{l}\langle\alpha_i,\beta\rangle\langle\alpha_j,\beta\rangle\langle\alpha_k,\beta\rangle x_i x_j x_k\,, \nonumber\\
    F^+_{\rm GW}= & 2\nu \sum_{\beta\in \mathcal{R}^+} \mathrm{Li}_3\l(\re^{-\langle\beta,h\rangle}\r) \,.
\end{align}
When $\RR = \mathrm{A}_l$, 
the same reasoning applied to the full two-dimensional $\mathsf{T}'$-action on $Z$ yields 
\cite{MR2510741,MR2861610}
\begin{align}
    F^0_{\rm GW}=& \frac{1}{6 \nu_1\nu_2}\frac{1}{|\mathsf{G}|}x_{l+1}^3 -\frac{1}{2}x_{l+1}\sum_{i,j=1}^l C_{ij} x_i x_j 
    - \frac{1}{6} \sum_{i,j,k,i',j',k'=1}^l 
        C_{ii'}C_{jj'}C_{kk'} \mathfrak{C}_{i'j'k'} x_i x_j x_k \,, \nn \\
    F^+_{\rm GW}=& (\nu_1 + \nu_2) \sum_{\beta\in \mathcal{R}^+} \mathrm{Li}_3\l(\re^{-\langle\beta,h\rangle}\r) \,.
    \label{GW_potential_typeA_fulltorus}
\end{align}
where the symmetric trilinear form $\mathfrak{C}_{ijk}$ is determined by its value at $i\leq j \leq k$ as
\[\mathfrak{C}_{ijk} = \frac{\left(j\nu_1 + (l+1-j) \nu_2\right) i(l+1-k)}{l+1}\,, \qquad i\leq j \leq k.\]

\subsubsection{Landau--Ginzburg mirrors from ADE spectral curves}
\label{sec:hurgen}

Let $N \in \bbZ_{>0}$ and  ${H}_{g,\mathsf{m}}$ be the moduli space 
of smooth genus $g$-covers of $\bbP^1$ with ramification profile at infinity specified by a partition~$\mathsf{m} \vdash N$.  We will write $\pi$, $\lambda$ and $\Sigma_i$ for, respectively, the universal
family, the universal map, and the sections marking $\{\infty_i\}_i \coloneqq \lambda^{-1}([1:0])$, as per the following commutative diagram:
\beq
\xymatrix{ C_g \ar[d]  \ar@{^{(}->}[r]& \cC_{g,\mathsf{m}}\ar[d]^\pi  \ar[r]^{\lambda}  &  \bbP^1 \\
                     [\lambda]  \ar@{^{(}->}[r]^{\rm pt}   \ar@/^1pc/[u]^{P_i}&                                             H_{g,\mathsf{m}}  \ar@/^1pc/[u]^{\Sigma_i}& 
}
\label{eq:hurdiag}
\eeq
We furthermore denote by $\rd=\rd_\pi$  the relative differential with respect to the
universal family and $p_i^{\rm cr}\in C_g\simeq \pi^{-1}([\lambda])$ the critical locus $\rd
\lambda =0$ of the universal map. By the Riemann existence theorem, the critical values of $\lambda$, \[(u_i)_{i = 1, \dots ,d_{g;\mathsf{m}}}\,,\] 
serve as local coordinates away from the discriminant locus 
    $u_i = u_j$ for $i \neq j$. 
On its complement, we can construct a family of semi-simple, commutative, $\bbC$-algebra structures on the tangent fibres at $(u_1, \dots , u_{d_{g,\mathsf{m}}})$ by stipulating that the coordinate vector fields in the $u$-chart are the idempotents of the algebra,
\begin{equation}
    \partial_{u_i} \cdot \partial_{u_j} = \delta_{ij}\partial_{u_i}.
\label{eq:prodss}
\end{equation}
Let $\mu : \cC_{g,\mathsf{m}} \to \bbP^1$ be a surjective morphism such that the
relative one-form \[\rd \log \mu \in \Omega^1_{\cC_{g,\mathsf{m}}/H_{g,\mathsf{m}}}(\infty_0+\dots +\infty_m)\] is an exact third-kind differential\footnote{This is a special case of a type III admissible differential, in the classification of \cite[Lect.~5]{Dubrovin:1994hc}.} on the fibres of the universal curve with simple poles at $\infty_i$, with residues \[\Res_{\infty_i} \rd\log\mu =a_i \in \bbZ\,.\] If $\mu$ does not factor through $\lambda$, this defines an Ehresmann connection on $TC_{g,\mathsf{m}}$ where points in nearby fibres of the universal curve are identified if they have the same image under $\mu$. This defines a meromorphic derivation 
\[
\delta_{\partial_{u_i}} ~ : ~  \mathrm{H}^0(\cC_{g,\mathsf{m}}, \cO_{\cC_{g,\mathsf{m}}})   \xrightarrow{\hspace*{.75cm}}    \mathrm{H}^0(\cC_{g,\mathsf{m}}, \cK_{\cC_{g,\mathsf{m}}})
\]
defined in local coordinates $p=(u_1, \dots , u_{d_{g;\mathsf{m}}}; \mu)$ for $C_{g,\mathsf{m}}$ as the partial derivative taken with respect to $u_i$ whilst keeping $\mu$ constant. 
A Frobenius manifold structure $\HH_{g,\mathsf{m}}^{[\mu]} \coloneqq (M_{\rm LG}, c_{\rm LG}, \eta_{\rm LG}, e)$ can be defined on the Hurwitz space $M_{\rm LG} \coloneqq H_{g,\mathsf{m}}$
by the residue formulas \cite{Dubrovin:1994hc}
\begin{align}
    \eta_{\rm LG}(X, Y) \coloneqq & \sum_{i} \underset{p_i^{\text{cr}}}{\text{Res}}\dfrac{\delta_X~\lambda~ \delta_Y\lambda}{\text{d}\lambda} \phi^2\,,
        \label{eq:etares}
        \\
    c_{\rm LG}(X, Y, Z) \coloneqq  & \sum_{i} \underset{p_i^{\text{cr}}}{\text{Res}}\dfrac{\delta_X\lambda~\delta_Y\lambda~\delta_Z\lambda}{ \text{d}\lambda}\phi^2\,,     \label{eq:cres}\\
    \eta^\flat_{\rm LG}(X, Y)  \coloneqq  & \sum_{i} \underset{p_i^{\text{cr}}}{\text{Res}}\dfrac{\delta_X \log \lambda~ \delta_Y \log \lambda}{\rd \log \lambda} \phi^2\,,
    \label{eq:gres}
\end{align}
where
\[
\phi \coloneqq \frac{\rd \mu}{\mu}\,.
\]
The universal map $\lambda$ and the relative differential $\phi$ are referred to as the {\it superpotential} and the {\it primitive form} of $\HH_{g,\mathsf{m}}$. \\



A general mirror symmetry construction of Frobenius submanifolds of Hurwitz spaces isomorphic to $\cM_{\rm AW}$, and the associated Dubrovin-dual Frobenius manifolds, was given in \cite{Brini:2021pix} (see also \cite{MR1606165,Dubrovin:2015wdx}), as we review here. For $(\RR,\widehat{\omega})$ be a marked ADE pair, let $\bar k \in \{1, \dots, l\}$ be such that $\omega_{\bar k} = \widehat{\omega}$. Fix $0 \neq \omega \in \Lambda_w^+$ a non-zero dominant weight. Starting from the characteristic polynomial of a regular element of $\cT$ in the representation $\rho_\omega$,
\beq
\cQ = \prod_{\omega'\in \cW(\omega)} \l(\re^{\bra\omega', x\ket }-\mu\r) \in \bbQ[Y_1, \dots , Y_l][\mu]\,,
\label{eq:weylrel}
\eeq
define 
\beq 
\cP\l(y_1, \dots , y_{l+1};\lambda,\mu\r) \coloneqq
 \cQ\l(Y_i=y_i\re^{-d_i y_{l+1}}-\delta_{i \bar k} \lambda \re^{-y_{l+1}};\mu \r)\,.
 \label{eq:shiftfun}
\eeq 
%
For fixed $y\in M^{\rm ss}_{\rm AW}$, \eqref{eq:shiftfun} defines a plane algebraic curve 
$C_y \coloneqq \bbV\l(\cP\r)\,.$
%
Let $\overline{C_y}$ denote the normalisation of the projective closure of the fibre at $y$, $g\coloneqq h^{1,0}\big(\overline{C_y}\big)$, and let $\mathsf{m}$ be the ramification profile over infinity of the 
Cartesian projection $\lambda :  \overline{C_y} \rightarrow \bbP^1$. The corresponding family is the pull-back of the universal curve to $M^{\rm ss}_{\rm AW}$, where the pull-back  metric, product, and intersection tensors are given by \eqref{eq:etares}--\eqref{eq:gres}.
 %
%
%
with $\{p_i^{\rm cr}\}_m$ the ramification points of $\lambda :
\overline{C_w} \longrightarrow \bbP^1$. 
For convenience and later comparison with $\cM_{\rm GW}$, it will be helpful to rescale the primitive differential and the linear coordinate on the second factor of $\mathfrak{h}^{\rm reg} \oplus \bbC$ as
\beq
\phi^2 \longrightarrow \frac{2\nu\dim_\bbC \mathfrak{g} }{\bra \omega, \omega+2 \mathsf{w}\ket \dim_\bbC \rho_\omega} \l(\frac{\rd \mu}{\mu}\r)^2\,, \qquad x_{l+1} \to \frac{1}{2\nu}\, x_{l+1}\,,
\label{eq:phiresc}
\eeq
where 
\beq 
\mathsf{w} \coloneqq \frac{1}{2}\sum_{\beta\in \RR^+}\beta = \sum_{i=1}^l \omega_i 
\label{eq:rho}
\eeq 
is the Weyl vector. By \eqref{eq:etares}--\eqref{eq:gres}, this just results in an overall rescaling of the prepotential for $\cM_{\rm LG}$.

\begin{thm}[Mirror symmetry for $\cM_{\rm AW}$]
For all marked pairs $(\RR, \widehat{\omega})$, the Landau--Ginzburg formulas \eqref{eq:etares}--\eqref{eq:gres} define a
semi-simple conformal Frobenius manifold \[\cM_{\rm
  LG}=(M_{\rm LG}, \eta_{\rm LG}, c_{\rm LG}, e,E)\,,\] with 
%
$e= y_{l+1}^{-1} \de_{\widehat{y}}$, $E= y_{l+1} \de_{y_{l+1}}$.
%
Furthermore,
%
\[\cM_{\rm LG} \simeq \cM_{\rm AW}\,.\]%
\label{thm:mirrorEAW}
\end{thm}


The next Proposition \cite{MR2070050,MR2287835} shows that, for $\cM\hookrightarrow \HH_{g,\mathsf{m}}$ a Frobenius submanifold of a Hurwitz space, the Dubrovin-dual Frobenius manifold $\cM^\flat$ is 
obtained by replacing the superpotential by its logarithm 
\[\lambda \longrightarrow \log \lambda\,.\]
\begin{prop}
Let $\iota: \cM = (M,c,\eta,e, E) \hookrightarrow \HH_{g,\mathsf{m}}$ be a semi-simple conformal Frobenius submanifold of a Hurwitz space defined by a Landau--Ginzburg pair $(\tilde{\lambda}, \tilde{\phi})$. Let $\cM^{\rm log}$ be the Frobenius submanifold structure on $M^{\rm ss}$ defined by the formulas 
\eqref{eq:etares}--\eqref{eq:cres} with $\lambda=\log \tilde{\lambda}$, $\phi=\tilde{\phi}$, and $X,Y, Z \in \iota_*(\cX(M))$. Then,
\[
\cM^{\rm log} \simeq \cM^\flat\,.
\]
\label{prop:duallog}
\end{prop}
\vspace{-.75cm}
From \cref{prop:duallog}, 
the Frobenius structure of 
$\cM_{\rm LG}^\flat \simeq \cM_{\rm AW}^\flat$ is given as
\beq
\label{eq:todalogeta}
\eta^\flat_{\rm AW}(X, Y) =  
\eta^\flat_{\rm LG}(X, Y) =   \sum_{m}\underset{p_m^{\text{cr}}}{\text{Res}}\frac{\delta_{X}\lambda~
  \delta_{Y}\lambda}{\lambda \rd \lambda}\phi^2\,,  
\eeq
\beq
c^\flat_{\rm AW}(X, Y, Z) =  
c^\flat_{\rm LG}(X, Y, Z) =
\sum_{m}\underset{p_m^{\text{cr}}}{\text{Res}}\frac{\delta_{X}\lambda~
  \delta_{Y}\lambda \delta_{Z}\lambda}{\lambda^2\rd \lambda}\phi^2\,.
\label{eq:todalogc}
\eeq

\section{Dubrovin duality and mirror symmetry} \label{DubDuality}

In this section we will state and prove our two main theorems relating the Dubrovin-dual Frobenius structures of $\cM_{\rm AW}$ and $\cM_{\rm LG}$ to the $\mathsf{T}$-equivariant quantum cohomology of $\mathrm{ADE}$ resolutions. 

\begin{thm}[Dubrovin duality for $\cM_{\rm AW}$ and $\cM_{\rm GW}$]
For all $(\RR, \widehat{\omega})$, we have
%
$\cM^\flat_{\rm AW} \simeq \cM_{\rm GW}$.
\label{thm:ddQH}
\end{thm}

\begin{thm}[Mirror symmetry for $\cM_{\rm GW}$]
For all $(\RR, \widehat{\omega})$, we have
%
$\cM^\flat_{\rm LG} \simeq \cM_{\rm GW}$.
\label{thm:msQH}
\end{thm}

\begin{rmk}
In \cref{thm:ddQH}, when comparing the Frobenius structures on $\cM_{\rm AW}^\flat$ and $\cM_{\rm GW}$, we shall need to formally set $\nu=1$: by the Degree Axiom, the dependence on $\nu$ is reinstated on the Gromov--Witten side (or introduced on the extended affine Weyl side) upon rescaling 
\[ F_{\rm GW} \to \nu F_{\rm GW}\,, \quad F^\flat_{\rm AW} \to \nu F^\flat_{\rm AW}\,, \quad x_{l+1} \to \frac{x_{l+1}}{\nu}\,.\]
\label{rmk:nu_dep}
\end{rmk}
\vspace{-.5cm}
In the following we will describe how \cref{thm:msQH} for rational spectral curves translates into a comparison statement between the Gromov--Witten calculation in \eqref{GW_potential} and \eqref{GW_potential_typeA_fulltorus} on one hand, and an explicit summation of residues which localise to the fibre over zero of the $\lambda-$projection on the other: we will compute this in closed-form for $\RR=\mathrm{A}_l$ and $\RR=\mathrm{D}_l$, thereby proving \cref{thm:msQH}, and therefore \cref{thm:ddQH}, for the classical series. We also explain in general how to reduce \cref{thm:ddQH} to a comparison statement of the structure constants for the Frobenius algebras restricted to a finite set of points on the complement of the discriminant: computing this explicitly for $\RR=\mathrm{E}_l$ will provide a proof of \cref{thm:ddQH}, and therefore \cref{thm:msQH}, for the three exceptional cases.

\subsection{Landau--Ginzburg mirror symmetry}
\label{sec:LGsum}
We start by looking at the argument of the residues in \eqref{eq:todalogeta},
\beq 
\Upsilon_{i,j,k}(p) \coloneqq
\frac{\delta^{(\mu)}_{\de_{x_i}} \lambda ~ \delta^{(\mu)}_{\de_{x_j}} \lambda \delta^{(\mu)}_{\de_{x_k}} \lambda}{\lambda^2 \mu^2  \de_\mu \lambda}\rd\mu(p),
\label{eq:ups}
\eeq
so that
\beq 
\eta^\flat_{\rm LG}(\de_{x_i}, \de_{x_j}) = \sum_{m}\underset{p=p^{\rm cr}_m}{\Res}\Upsilon_{i,j,l+1}(p)\,, \quad 
c^\flat_{\rm LG}(\de_{x_i}, \de_{x_j}, \de_{x_k}) = \sum_{m}\underset{p=p^{\rm cr}_m}{\Res}\Upsilon_{i,j,k}(p)\,.
\label{eq:gammaresups}
\eeq 
From \eqref{eq:ups}, we deduce that the pole structure of $\Upsilon_{i,j,k}(p)$ is as follows:
\ben[(i)]
\item it has at most simple poles at the critical points $\{p_l^{\rm cr}\}$, for which $\rd\lambda(p_i^{\rm cr})=0$;
\item it has a pole of order at most \[\max \{2-\delta_{i,l+1}-\delta_{j,l+1} -\delta_{k,l+1},0\}\] at $\lambda(p)=0$: this follows from \[ \delta^{(\mu)}_{\de_{x_{l+1}}} \lambda \propto \lambda\,,\] 
which offsets a linear power in the vanishing of the denominator;
\item it has at most simple poles at $\mu(p)=0$ (when $\lambda(p)=\infty$) when $i=j=k=l+1$;
\item it has a pole of order at most 
\[ \max\{1-\delta_{i,l+1}-\delta_{j,l+1} -\delta_{k,l+1},0\}\]
at the critical points $\{q^{\rm cr}_m\}$ of the $\mu$-projection, $\rd\mu(q^{\rm cr}_m)=0$. These are the loci where the Ehresmann connection induced by the $\mu$-foliation is singular and $\delta^{(\mu)}_{\de_{x_i}} \lambda$ possibly develops a pole. These singularities are partially offset by a vanishing of the same order of $\rd \mu/{\de_\mu \lambda}$.
\een  
The residue sums \eqref{eq:gammaresups} pick up the contributions from the residues of type (i) alone: the difficulty in writing the critical points of the superpotential as algebraic functions of $(x_1, \dots, x_l)$ makes, however, their individual comuputation unwieldy. To overcome this problem, we turn the contour around and equate the sum of residues at the critical points in \eqref{eq:gres} to a much more manageable sum of residues at poles and zeros of $\mu$ and $\lambda$ (type (ii) and (iii)) as well as a contribution from residues at critical points of $\mu$ (type (iv)). 

Computing the individual contributions from type (iv) residues is as difficult, if not more, than computing those arising from the critical points of $\lambda$. However, there are two scenarios when they can be shown to vanish identically. 

\begin{itemize}
    \item If any of $i$, $j$ or $k$ is equal to $l+1$, 
we have
\[
\mathrm{ord}_{q_m^{\rm cr}} \Upsilon_{i,j,k} = 
\max\{1-\delta_{i,l+1}-\delta_{j,l+1} -\delta_{k,l+1},0\}=0\,
\]  
hence there is no pole at the ramification points of the $\mu$-projection. The sum was calculated explicitly for all extended affine Weyl Frobenius manifolds with canonically marked node in \cite[Thm.~3.5]{Brini:2021pix}, showing that 
\[
\eta^\flat_{\rm LG}(\de_{x_i}, \de_{x_j}) = \sum_{m}\underset{p=p^{\rm cr}_m}{\Res}\Upsilon_{i,j,l+1}(p) = \eta_{\rm GW}(\de_{x_i}, \de_{x_j})\,
\]
as expected.
    \item If $\deg_\lambda \cP = \deg_{\widehat{Y}} \cQ =1$, implying that the spectral curve is rational, we have $\{q_m^{\rm cr}\} = \emptyset$, so once again the sum over residues only picks up contributions from zeros and poles of $\mu$ and $\lambda$. For $i,j,k \neq l+1$, the sole contribution arises from the zeroes of the rational function $\lambda(\mu)$.  
\end{itemize}

When $\RR=\mathrm{A}_l$ or $\RR=\mathrm{D}_l$, it was shown in \cite{Brini:2021pix} that $\deg_\lambda \cP=1$, and for
these two cases the summation over residues can be performed explicitly\footnote{Using the superpotential derived in \cite{Dubrovin:2015wdx} the almost dual prepotentials for the $BCD$-cases with arbitrary marked node may easily be calculated. These are of the same form, corresponding to a $BC_l$-root system.}.

\subsubsection{Proof of \cref{thm:msQH} for the $\mathrm{A}_l$ series with arbitrary marked weight}

Let $\omega$ be the highest weight of the defining representation $\rho$ of $\mathfrak{g}=\mathrm{sl}_{l+1}(\bbC)$,  and let  $1\leq \bar{k} \leq l$ be a choice of marked node.  From \eqref{eq:shiftfun}, the corresponding superpotential reads

\beq
    \lambda 
    = \frac{y_{l+1} \prod_{\omega' \in \cW(\omega)}
    (\re^{\bra \omega', h\ket}-\mu)}{(-\mu)^{l+1-\bar{k}}}
    = (-1)^{l+1 - \bar{k}} y_{l+1} \prod_{j=1}^{l}\kappa_j^{-\frac{l+1 - \bar{k}}{l+1}}\frac{(1-q)\prod_{k=1}^l(1-\kappa_k q)}{q^{l+1-\bar{k}}},
    \label{eq:lambdaA}
\eeq
where 
\[
\kappa_j \coloneqq \prod_{i=j}^l \re^{- p_i}\, , \, 1\leq j \leq l\,, \quad q \coloneqq \re^{-\langle\omega_{\rm min},h\rangle}\mu\,, \quad 
\omega_{\rm min} \coloneqq -\frac{1}{l+1}\sum_{i=1}^{l}i\alpha_i\,.
\]
The structure constants of the Frobenius manifold structure on $\cM_{\rm LG}^\flat$ were computed 
using the method described in \cref{sec:LGsum} in \cite{MR2287835} and, in a slightly generalised fashion, in \cite[Thm.~5.4]{Brini:2013zsa}, where more generally
the three-point functions of the full $\mathsf{T}'$-equivariant Gromov--Witten theory of $Z$ were seen to equate\footnote{More precisely, the comparison in \cite{Brini:2013zsa} was performed for $Z\times \mathbb{C}$ with a 2-torus action acting with opposite weights on the canonical bundle of the two factors. The corresponding genus-zero GW potential differs from that of $Z$ by by an overall weight factor of $-(\nu_1 + \nu_2)$.} those arising from the Landau--Ginzburg pair $(\log\lambda, \phi)$, with
\[
\lambda =\re^{\frac{x_{l+1}}{\nu_1+\nu_2}}\prod_{j=1}^{l}\kappa_j^{-\frac{\nu_1}{\nu_1 + \nu_2}}\frac{(1-q)\prod_{k=1}^l(1-\kappa_k q)}{q^{(l+1)\frac{\nu_1}{\nu_1 + \nu_2}}}\,, \quad
\phi^2 = (\nu_1 + \nu_2)\l(\frac{\rd q}{q}\r)^2.
\]
Taking the restriction to the one-dimensional subtorus $\mathsf{T}$ that fixes the irreducible component $E_{\overline{k}}$ as in \eqref{eq:ZxCtoZ},
the superpotential $\lambda$ coincides with \eqref{eq:lambdaA}, hence
\[
c_{\rm GW}(\varphi_i, \varphi_j, \varphi_k)
= c^\flat_{\rm LG}(\de_{x_i},\de_{x_j},\de_{x_k})\,,
\]
which yields the statement of \cref{thm:msQH} for $\RR=\mathrm{A}_l$ and $\widehat{\omega}=\omega_{\bar k}$. 

\subsubsection{Proof of \cref{thm:msQH} for the $\mathrm{D}_l$ series}

Let $\omega$ be the highest weight of the defining representation of $\mathfrak{g}=\mathrm{so}_{2l}(\bbC)$. From \eqref{eq:shiftfun}, the corresponding superpotential
reads
\beq
    \lambda=\frac{\kappa_{l+1} \prod_{\omega' \in \cW(\omega)}
    (\re^{\bra \omega', x\ket}-\mu)}{\mu^{l-2}(\mu^2-1)^2}=\frac{\kappa_{l+1} 
    \prod_{i=1}^l(\mu-\kappa_i)(\mu-\kappa_i^{-1})}{\mu^{l-2}(\mu^2-1)^2}
\label{eq:lambdaD}
\eeq
%
where
\beq
\log\kappa_i = \sum_{j=1}^lG_{ij}x_j\,,  1\leq i \leq l\,,
\label{kappa_to_yi}
\eeq

and 
    $G_{ij} \coloneqq \delta_{ij}-\delta_{i,j+1} + \delta_{i,l-1}\delta_{j,l}$;
note that
    $(G^T G)_{ij}=C_{ij}\,.$
As in \eqref{eq:phiresc}, we shall normalise the coordinate $x_{l+1}$ and the quadratic differential associated to the primitive form such that:
\beq
    \log\kappa_{l+1} = \frac{x_{l+1}}{2 \nu}\,, \qquad \phi^2=\nu \l(\frac{\rd\mu}{\mu}\r)^2\,.
\eeq


Since the spectral curve is rational in this case and the $\mu$-projection is unramified, to perform  the sum of residues over critical points in \eqref{eq:todalogeta}--\eqref{eq:todalogc} we the contour of integration around on the fibres of the mirror family, and instead sum over residues at the other poles of the integrand. From \eqref{eq:lambdaD}, these are located at the support of the divisor $(\lambda)$: the locus of zeroes and poles of $\lambda$. We thus need to show that

\beq
c_{\rm GW}(\varphi_i, \varphi_j, \varphi_k) = c^\flat_{\rm LG}(\varphi_i, \varphi_j, \varphi_k) = -\sum_{p\in \mathrm{supp}(\lambda)}\mathop{\Res}_p \frac{\delta_{\de_{x_i}}(\lambda) \delta_{\de_{x_j}}(\lambda)\delta_{\de_{x_k}}(\lambda)}{\lambda^2 \rd \lambda}\phi^2.
\label{eq:GW=LG}
\eeq
Write 
\beq
R_{ijk}\coloneqq \frac{1}{2\nu}c^\flat_{\rm LG}\l(\frac{\partial}{\partial \log \kappa_i},\frac{\partial}{\partial \log \kappa_j},\frac{\partial}{\partial \log \kappa_k}\r)\,, 
\label{eq:Rijk}
\eeq
and $R_{ijk}^{[q]}$ for the contribution of the residue  at $\mu=q$ in the sum \eqref{eq:GW=LG}. We 
start with the following 
\begin{lem} 
We have
\[
    R_{ijk}^{[1]}=  R_{ijk}^{[-1]}=0\,,\quad 
    R_{ijk}^{[0]}=  R_{ijk}^{[\infty]}= \frac{\delta_{i,l+1}\delta_{j,l+1} \delta_{k,l+1}}{2(l-2)}\,,
    \]
\begin{align*}
& R_{ijk}^{[\kappa_m]} + R_{ijk}^{[1/\kappa_m]} = (\delta_{i,j,m}+\delta_{j,k,m}+\delta_{i,k,m})  \times  \bigg\{ \frac{\delta_{i,j,k} q_i}{3}  + (1-\delta_{i,j,k})
 \Big[ -(\delta_{i,j,m}\delta_{k,l+1}+\delta_{j,k,m}\delta_{i,l+1}+\delta_{i,k,m}\delta_{j,l+1}) \\
&  + (1-\delta_{i,j,m}\delta_{k,l+1}-\delta_{j,k,m}\delta_{i,l+1}-\delta_{i,k,m}\delta_{j,l+1})(\delta_{i,j,m} p_{mk} +\delta_{k,i,m} p_{mj} +\delta_{j,k,m} p_{mi}) 
    \Big] \bigg\}\,,
\end{align*}
with 
\[ p_{ij} =
\frac{\kappa _i \left(\kappa _j^2-1\right)}{\left(\kappa
   _i-\kappa _j\right) \left(\kappa _i \kappa _j-1\right)}\,,
\qquad 
q_k = \sum_{n \neq k} \frac{\kappa_n(1- \kappa _k^2)}{\left(\kappa _k-\kappa
   _n\right) \left(\kappa _k \kappa _n -1\right)}\,.
\]
\label{lem:resD}
\end{lem}
\begin{proof}
The statement follows from a lengthy, but straightforward calculation of the rational residues in 
\beq 
R_{ijk}^{[q]}=- \mathop{\Res}_{\mu=q}\frac{\prod_{m \in \{i,j,k\}}\l(\frac{-\kappa_m}{\mu-\kappa_m}+\frac{\kappa_m^{-1}}{\mu-\kappa_m^{-1}}\r)^{1-\delta_{m,l+1}}
}{\sum_{r=1}^l(\frac{1}{\mu-\kappa_r}+\frac{1}{\mu-\kappa_r^{-1}})-\frac{l-2}{\mu}-\frac{4\mu}{\mu^2-1}}\frac{d\mu}{2\mu^2}\,,
\eeq
for $q=0$, $\infty$, $\pm 1$, $\kappa_m^{\pm}$; we omit the details here. 
\end{proof}

\begin{cor}
\label{cor:mirDclass}
We have, for $1\leq i,j \leq l+1$,
\begin{align}
c^\flat_{\rm LG}\l(\frac{\partial}{\partial x_{i}},\frac{\partial}{\partial x_{j}},\frac{\partial}{\partial x_{l+1}}\r)= 
\begin{cases} 
\frac{1}{4\nu^2(l-2)}\,, & i=j=l+1\,, \\
-C_{ij}\,, & i, j \neq l+1\,, \\
0 & {\rm else}\,.
\end{cases}
\label{eq:mirDclass}
\end{align}
\end{cor}

\begin{proof}
Since $\log\kappa_{l+1}=x_{l+1}/(2\nu)$, the l.h.s. 
for $i=j=l+1$ is 
\[
\frac{1}{4\nu^2}R_{l+1,l+1,l+1}=\frac{R_{l+1,l+1,l+1}^{[0]}+R_{l+1,l+1,l+1}^{[\infty]}}{4\nu^2}=\frac{1}{4\nu^2(l-2)}\,.
\]
The vanishing when $i\neq j$ and either $i=l+1$ or $j=l+1$ is immediate from \cref{lem:resD}.  For the case $1\leq i,j \leq l$, we use \eqref{kappa_to_yi} to get
\begin{align*}
c^\flat_{\rm LG}\l(\frac{\partial}{\partial x_{l+1}},\frac{\partial}{\partial x_i},\frac{\partial}{\partial x_j}\r) 
&=\sum_{k,m,n=1}^{l} G_{ki} \Big(R_{k,m,l+1}^{[\kappa_n ]}+R_{k,m,l+1}^{[1/\kappa_n ]}\Big)G_{mj} =-\sum_{k=1}^{l} G_{ki} G_{kj} =-C_{ij}\,.
\end{align*}
\end{proof}
Recall that,
    w.r.t. the orthonormal basis $\{\epsilon_i\}_{i=1}^l$ of $\mathbb{R}^l$, the $\mathrm{D}_l$ root system is
    \begin{align*}
    \Pi = \{\epsilon_i-\epsilon_{i+1}\}_{i=1}^{l-1} \cup \{\epsilon_{l-1}+\epsilon_l\}\,, \quad 
        \RR^+=\{\epsilon_i\pm \epsilon_j\}_{i<j}\,, \quad \bra \alpha_i, \epsilon_j \ket= G_{ji}\,.
    \end{align*}
\begin{defn}\label{defn_Omega}
Choose a bijection \[\sigma:\Big\{(i,j)\,\big|\,1\leq i<j\leq l\Big\}\longleftrightarrow \bigg\{1,\dots,\frac{l(l-1)}{2}\bigg\}\,,\] and for each pair $(i,j)$ with $1\leq i<j\leq l$, define 
$\Theta^\pm$ and 
$\Theta$ by
\[
  \Theta^+_{\sigma(i,j),k} \coloneqq
  \begin{cases}
    1 & \text{if $k=i$ or $k=j$}\,, \\
    0 & \text{otherwise}\,,
  \end{cases}\qquad
  \Theta^-_{\sigma(i,j),k} \coloneqq
  \begin{cases}
    1 & \text{if $k=i$}\,, \\
    -1 & \text{if $k=j$}\,, \\
    0 & \text{otherwise}\,,
  \end{cases}\qquad 
\Theta \coloneqq
\left(
\begin{matrix}
    \Theta^+\\
    \Theta^-
\end{matrix}
\right)\,.
\]

\end{defn}
\begin{cor}
Let $\epsilon$ be the column vector $(\epsilon_1,\dots,\epsilon_l)^T$. Then the rows of $\Theta \cdot \epsilon$ give all the positive roots of $\mathrm{D}_l$, 
\begin{equation}\label{positive_weights}
        \left(
    \beta_1, \dots
    \beta_{l(l-1)}
\right)
=\Theta \epsilon\,.
\end{equation}
\end{cor}
Define now coordinates $(\tau_1,\dots,\tau_l)$ of $\mathrm{H}^2(Z, \bbC)$ via 
\begin{equation}\label{x_to_tau}
(x_1, \dots, x_l)^T
=C^{-1}G^T
(\tau_1, \dots, \tau_l)^T
\,.    
\end{equation}
By \eqref{kappa_to_yi}, we have:
\begin{equation*}
(\log\kappa_1, \dots, \log\kappa_l)^T = G (x_1, \dots, x_l)^T = G C^{-1}G^T
(\tau_1, \dots, \tau_l)^T
\end{equation*}
Using $G^T G=C$, we get $GC^{-1}G^{T}=I$, hence $\log \kappa_i=\tau_i$, $1\leq i\leq l$.

\begin{lem}\label{prop:FGWDplus}
The positive degree part of the genus Gromov--Witten primary potential of $Z$ is
\begin{equation}
    F_{\rm GW}^+=2\nu\sum_{i=1}^{l(l-1)}\mathrm{Li}_3\l(\exp\l(\sum_{j=1}^l -\Theta_{i j}\tau_j\r)\r)\,.
\end{equation}
\end{lem}
\begin{proof}
In term of the matrix $A_{ij}=\bra \beta_i, \alpha_j \ket$ of coefficients of the positive roots $(\beta_1, \dots, \beta_{l(l-1)})$ in the Omega basis,
\begin{equation}
\RR^+ 
=A \Omega\,,
\end{equation}
we find, using \eqref{x_to_tau}, 
\[ F_{\rm GW}^+  = 2\nu \sum_{\sigma =1}^{l(l-1)}\mathrm{Li}_3\l(\exp\l(\sum_{j=1}^l -(AC^{-1}G^T)_{\sigma j}\tau_j\r)\r)\,.
\]
On the other hand, expressing fundamental weights in the orthonormal basis, 
\begin{equation*}
    \left(
    \beta_1, \dots, 
    \beta_{l(l-1)}
    \right)^T
=AC^{-1}
\left(
     \alpha_1,
     \dots, 
    \alpha_l
\right)^T
=AC^{-1}G^T
\left(
    \epsilon_1,
    \dots,
    \epsilon_l
\right)^T\,. 
\end{equation*}
From this we deduce that $\Theta=AC^{-1}G^T$ by comparing to \eqref{positive_weights}, thereby proving the claim.
\end{proof}
\begin{proof}[Proof of \cref{thm:msQH} for $\RR=\mathrm{D}_l$]
By \cref{cor:mirDclass}, 
we need only consider the case
$1\leq i,j,k \leq l$. From \cref{prop:FGWDplus}, we get
\begin{align}
    &\frac{1}{2\nu}\frac{\partial^3 F_{\rm GW}^+}{\partial \tau_i \partial \tau_j \partial \tau_k}=-\sum_{\sigma=1}^{l(l-1)}\Theta_{\sigma i}\Theta_{\sigma j}\Theta_{\sigma k}\frac{\exp(-\sum_{m=1}^l\Theta_{\sigma m}\tau_m)}{1-\exp(-\sum_{m=1}^l\Theta_{\sigma m}\tau_m)},\label{3-point_Omega}
\end{align}
where $\sigma$ is an element of the index set $\sigma \in \{(m,n)|1\leq m<n\leq l\}$, and consider the case $k=j=i$ for starters. From \cref{defn_Omega}, this reads explicitly as 
\begin{align*}
\frac{1}{2\nu}\frac{\partial^3 F^+_{\rm GW}}{\partial \tau_i \partial \tau_i \partial \tau_i} =
-\sum_{k\neq i}\frac{\exp(-\tau_i-\tau_k)}{1-\exp(-\tau_i-\tau_k)}
    -\sum_{k>i}\frac{\exp(-\tau_i+\tau_k)}{1-\exp(-\tau_i+\tau_k)}
    + \sum_{k<i}\frac{\exp(-\tau_k+\tau_i)}{1-\exp(-\tau_k+\tau_i)}\,.
    \label{eq:3GWiii}
\end{align*}
Likewise, for $k=j\neq i$, we spell out \eqref{3-point_Omega} to be
\[
\frac{1}{2\nu}\frac{\partial^3F^+_{\rm GW}}{\partial \tau_i \partial \tau_i \partial \tau_j}=-\frac{\exp(-\tau_i-\tau_j)}{1-\exp(-\tau_i-\tau_j)}+\frac{\exp(-\tau_i+\tau_j)}{1-\exp(-\tau_i+\tau_j)} + \frac{1}{2}(\mathrm{sgn}(i-j)+1)\,.
\]
As, for fixed $\sigma$, $\Theta_{\sigma i}$  is only non-zero for exactly two values of $i$, we finally have that, for $i,j,k$ all distinct,
\[
\frac{\partial^3F^+_{\rm GW}}{\partial \tau_i \partial \tau_j \partial \tau_k} = 0\,. 
\]
Let us compare the above back to the structure constants of the Landau--Ginzburg product \eqref{eq:Rijk}.
From \cref{lem:resD} and for $i,j,k$ all distinct, we find
\beq
R_{iii}=\frac{1}{2\nu}\frac{\partial^3F^+_{\rm GW}}{\partial \tau_i \partial \tau_i \partial \tau_i}-l+i 
    \,, \quad
    R_{iij}=\frac{1}{2\nu}\frac{\partial^3 F^+_{\rm GW}}{\partial \tau_i \partial \tau_i \partial \tau_j}-\frac{1}{2}(\mathrm{sgn}(i-j)+1) 
    \,,\quad
R_{ijk}=\frac{1}{2\nu}\frac{\partial^3 F^+_{\rm GW}}{\partial \tau_i \partial \tau_j \partial \tau_k}=0\,.
\label{eq:RvsGW}
\eeq
Hence \cref{thm:msQH} for type $\mathrm{D}_l$ reduces to verifying the following identities relating the additive terms in the r.h.s. of \eqref{eq:RvsGW} to the $\mathsf{T}$-equivariant triple intersection numbers of $Z$:
\[
\sum_{\sigma=1}^{l(l-1)}\Theta_{\sigma i}^2\Theta_{\sigma j}=
\begin{cases}
2(l- i)\,, & 1\leq i=j\leq l\,, \\
\mathrm{sgn}(i-j)+1\,, & 1\leq i\neq j\leq l\,.
\end{cases}
\]
These are easily verified using the definition of $\Theta$, 
    concluding the proof.
\end{proof}

\subsection{Dubrovin duality via initial conditions}
\label{sec:dd}

A direct proof of \cref{thm:ddQH} would entail showing that the structure constants of $\cM_{\rm AW}^{\flat}$, expressed in terms of those of $\cM_{\rm AW}$ through the duality relation \eqref{eq:AD_formula}, coincide with the quantum cohomology product  \eqref{eq:cBG} of $Z$. Explicitly, in our case \eqref{eq:AD_formula} reads
\begin{equation}
    \frac{\partial^3 F_{\rm GW}}{\partial x_{i}\partial x_{j}\partial x_{k}}(x) = \sum_{a,b,A,B,C}(\eta^\flat_{\rm AW})_{ia} (\eta^\flat_{\rm AW})_{jb}\frac{\partial t_C}{\partial x_k}\frac{\partial x_a}{\partial t_A}\frac{\partial x_b}{\partial t_B} \l(c_{\rm AW}\r)^{AB}_{C}(t(x))\,,
    \label{eq:AD2}
\end{equation}
where 
\beq
\l(c_{\rm AW}\r)^{AB}_C = \sum_{M,N=1}^{l+1} \eta^{AM}_{\rm AW} \eta^{BN}_{\rm AW} \frac{\de^3 F_{\rm AW}}{\de t_M \de t_N \de t_C}\,.
\label{eq:cflat}
\eeq
\\
For a given marked pair $(\RR, \widehat{\omega})$,  \eqref{eq:AD2} could \emph{a priori} be proved by brute-force, as both sides of the equality are calculable, at least in principle. For the r.h.s., closed-form expressions for the prepotential $F_{\rm AW}(t)$ and the Saito flat coordinates $\{t_A(x)\}$ in linear coordinates on $\mathfrak{h}\times \bbC$
were found in \cite{MR1606165,Brini:2017gfi,Brini:2021pix};  for the l.h.s., the Gromov--Witten quantum prepotential is given by \eqref{eq:cBG}. \\ In practice, however, this approach is largely unfeasible. The entries of the Jacobian matrix $\partial_{x_k} t_C(x)$ are trigonometric polynomials with e.g. up to billions of terms for $\RR=\mathrm{E}_8$; for the same reason, computing the inverse $\partial_{t_A} x_a$ relevant to \eqref{eq:AD2} is well out of reach of symbolic computation packages, such as \emph{Mathematica}, even for very low values of $l$. \\

As we will show, both sides of \eqref{eq:AD2} are uniquely determined by their values at a (small) finite number of points in the semi-simple locus $\cM^{\rm ss}_{\rm GW} = \mathfrak{h}^{\rm reg} \oplus \bbC$: we will refer to thid as a \emph{set of initial conditions} for the Dubrovin duality. The functional equality \eqref{eq:AD2} reduces then to a \emph{numerical} equality over the set of initial conditions, which can in turn be verified very effectively. As explained in \cref{rmk:nu_dep}, we will formally set $\nu=1$ throughout this Section.

\subsubsection{Reduction to initial conditions}
We will start by establishing various homogeneity properties for the tensors appearing in \eqref{eq:AD2}. Following \cite{MR1606165}, we define $d_{l+1} \coloneqq \deg(y_{l+1})=0$. On the complement of the discriminant, the coefficients $(\eta^\flat_{\rm AW})^{\alpha\beta}$ of \eqref{eq:EAWmaniLG} in the chart parametrised by the coordinates $(y_1, \dots, y_{l+1})$ satisfy \cite[Lem.~2.1]{MR1606165}
\[
\deg (\eta^\flat_{\rm AW})^{\alpha \beta} = d_\alpha + d_\beta\,,
\]
while the inverse Gram matrix of the Saito metric is, by definition, 
\[\eta_{\rm AW}^{\alpha \beta}(y) = \frac{\partial (\eta^\flat_{\rm AW})^{\alpha \beta}}{\partial \widehat{y}}\,.
\]   
By \cite[Cor.~2.5]{MR1606165}, the flat coordinates $\{t_A\}_{A =1, \dots , l}$ are polynomials in $y_1,\dots,y_l,\re^{y_{l+1}}$ with  $\deg t_A = d_A$. Furthermore, from \cite{B02}, the determinant of the Jacobian matrix 
associated to the change-of-variables $Y_i \longrightarrow Y_i(x_1, \dots , x_l)$ is  
\beq
\delta \coloneqq \prod_{\beta\in R^+}\left(e^{\langle\beta, h\rangle/2}-e^{-\langle\beta, h\rangle/2}\right).
\eeq
By definition, $\delta$ is anti-invariant under the Weyl group action, 
    \[\delta(w h) = (-1)^{l(w)} \delta(h).\]
By orthogonal extension to $\mathfrak{h} \times \bbC$, the determinant of the Jacobian matrix associated to the change-of-variables $y_i \longrightarrow y_i(x_1, \dots , x_{l+1})$
is  
\beq 
\Delta := c\re^{(d_1 + \dots + d_l)x_{l+1}}\delta\,,
\label{disc}
\eeq
and $\Delta$ is therefore, for the same reason as $\delta$, anti-invariant w.r.t. the Weyl group action. 

\begin{lem}
    $\Delta^2\in \mathcal{A}$. Moreover, $\Delta^2$ is quasi-homogeneous of degree $2(d_1 + \dots + d_l)$.
    \label{lem:Delta2inA}
\end{lem}
\begin{proof}
We first show the invariance of $\Delta^2$ under the extended affine Weyl group $\widetilde{\mathcal{W}}$. It suffices to check it on the generators, given by the Weyl reflections, the co-root lattice generators, and an extra translation 
\beq
(h,x_{l+1})\longrightarrow (h + 2\pi i \widehat{\omega}, x_{l+1} - 2 \pi i)\,.
\label{eq:afftr}
\eeq
Obviously, since $\Delta$ is anti-invariant, $\Delta^2$ is a Weyl group invariant. For the invariance under the other generators of $\widetilde{\mathcal{W}}$, we use 
\beq
\delta = \re^{\langle \mathsf{w},h\rangle}\prod_{\beta\in \RR^+}\left( 1- \re^{-\langle\beta,h\rangle}\right)\,,
\label{delta_expression}
\eeq
where $\mathsf{w}$ is the Weyl vector \eqref{eq:rho}. An affine translation by a co-root $\alpha^{\vee}$,
\[h \longrightarrow h + 2\pi \ri \alpha^{\vee}\,,\]  will leave $\delta$ (and therefore $\Delta$ and $\Delta^2$) invariant, since  $\langle\mathsf{w},\alpha^{\vee}\rangle$ and $\langle\beta,\alpha^{\vee} \rangle$, $\beta\in \RR^+$ are integers. It only remains to consider the effect of the $\bbZ$-action generated by \eqref{eq:afftr}: under this translation, 
we have
\[\delta \longrightarrow \re^{2\pi i \langle \mathsf{w}, \widehat{\omega}\rangle} \delta\,, \qquad \re^{(d_1 + \dots + d_l)x_{l+1}} \longrightarrow \re^{-2 \pi  i (d_1 + \dots + d_l)} \re^{(d_1 + \dots + d_l)x_{l+1}}\,. \]
Hence 
$\Delta$
is invariant, using \eqref{eq:rho} and the fact that
\beq 
\langle\mathsf{w},\widehat{\omega}\rangle = \sum_{i=1}^l d_i\,,
\label{eq:rhodi}
\eeq
since $d_i = \langle \omega_i,\widehat{\omega}\rangle$. To conclude that $\Delta^2\in\mathcal{A}$, it remains to prove boundedness of $\Delta$ in the limit \eqref{limit}. From \eqref{eq:rhodi}, the restriction of $\Delta$ to the locus in \eqref{limit} is
\[
c \re^{(d_1 + \dots + d_l)x_{l+1}} \re^{\langle \mathsf{w},h\rangle}\prod_{\beta\in \RR^+}\left( 1- \re^{-\langle\beta,h\rangle}\re^{-\langle\beta,\widehat{\omega} \rangle \tau}\right).
\]
This is bounded when $\tau \to +\infty$, since $\langle\beta,\widehat{\omega}\rangle \in \bbZ_{>0}$. It is finally immediate to see that $\Delta^2(x)$ is quasi-homogeneous of the claimed degree, since it is monomial in $\re^{x_{l+1}}$ with exponent $2(d_1 + \dots + d_l)$.   
\end{proof}

Define now $(2,1)$-tensors \[\ell_{\rm GW} \in \mathrm{H}^0\big(M_{\rm GW}, \mathrm{Sym}^2 \cT_{M_{\rm GW}} \otimes \cT^*_{M_{\rm GW}}\big)\,, \quad 
\ell_{\rm AW} \in \mathrm{H}^0\big(M_{\rm AW}, \mathrm{Sym}^2 \cT_{\cM_{\rm AW}} \otimes \cT^*_{\cM_{\rm AW}}\big)\,,
\]
by the following expressions in the chart $(y_1, \dots, y_{l+1})$:
\begin{align}
\l(\ell_{\rm GW}\r)^{\alpha\beta}_{\epsilon} \coloneqq & \sum_{a,b,i,j,k}\frac{\partial^3F_{\rm GW}}{\partial x_i \partial x_j \partial x_k}(\eta^\flat_{\rm AW})^{ia}\frac{\partial y_\alpha }{\partial x_a} (\eta^\flat_{\rm AW})^{jb}\frac{\partial y_\beta }{\partial x_b} \frac{\partial x_k}{\partial y_{\epsilon}} \,, \nn \\
\l(\ell_{\rm AW}\r)^{\alpha\beta}_{\epsilon} \coloneqq &
\sum_{L,M,N}\frac{\partial y_{\alpha}}{\partial t_L}\frac{\partial y_{\beta}}{\partial t_M}\frac{\partial t_N}{\partial y_{\epsilon}}(c_{\rm AW})^{LM}_N\,.
\label{eq:lhsrhs}
\end{align}
%
\begin{prop}
We have
  %
\beq 
\Delta^2 \l(\ell_{\rm GW}\r)^{\alpha\beta}_{\epsilon}
\in \cA \,, \qquad  \l(\ell_{\rm AW}\r)^{\alpha\beta}_{\epsilon}
\in \cA \,.
\label{eq:fourpol}
\eeq 
\label{prop:fourpol}
\end{prop}
\begin{proof}
\vspace{-.75cm}
By 
\eqref{eq:Yi} and \eqref{eq:yi},
the entries of the Jacobian matrix (resp. its inverse),
\[ 
\cJ_{\alpha i} \coloneqq \frac{\partial y_{\alpha}}{\partial x_i}\,, \quad \mathrm{resp.}~ \cJ^{-1}_{i\alpha} = \frac{\partial x_i}{\partial y_{\alpha}}\,,
\]
are Laurent polynomials (resp. rational functions) in $(\re^{x_1},\dots, \re^{x_{l+1}})$. Since $\Delta$ is the determinant of the Jacobian matrix,  the expression
\[\Delta \frac{\partial x_k}{\partial y_{\epsilon}} \in \bbC[\re^{\pm x_1}, \dots, \re^{\pm x_{l+1}}] \] is again a Fourier polynomial in $(x_1,\dots, x_{l+1})$. Moreover, from \eqref{GW_potential}, 
the triple derivatives 
of the Gromov--Witten prepotential 
are rational functions in $\re^{\pm x_1},\dots, \re^{\pm x_{l+1}}$ with at most first order poles at 
$\delta=0$. Hence, 
\[\Delta^2 \l(\ell_{\rm GW}\r)^{\alpha\beta}_{\epsilon} \in \bbC[\re^{\pm x_1}, \dots, \re^{\pm x_{l+1}}] \]
is a Fourier polynomial in $(x_1,\dots, x_{l+1})$, as claimed. By \cref{lem:Delta2inA}, it remains to check that
the tensor components
$(\ell_{\rm GW})_\epsilon^{\alpha\beta}$  
are $\widetilde{\cW}$-invariant and bounded in the limit \eqref{limit}.
Let
 \beq
g^{\eta\alpha} \coloneqq \sum_{a,b}
\frac{\partial y_{\eta}}{\partial x_a}(\eta^\flat_{\rm AW})^{ab}\frac{\partial y_{\alpha}}{\partial x_b}\,.
\label{eq: g20iny}
\eeq

By \cite{MR1606165}, $g^{\eta\alpha}$ are elements of $\mathcal{A}$. Therefore, noticing that
\begin{align}
        \l(\ell_{\rm GW}\r)^{\alpha\beta}_{\epsilon}&= 
\sum_{i,j,k,a,b}        
        \frac{\partial^3F_{\rm GW}}{\partial x_i \partial x_j \partial x_k}(\eta^\flat_{\rm AW})^{ia}\frac{\partial y_\alpha }{\partial x_a} (\eta^\flat_{\rm AW})^{jb}\frac{\partial y_\beta }{\partial x_b} \frac{\partial x_k}{\partial y_{\epsilon}} \nonumber\\
        &= \sum_{i,j,k,\delta,\eta}   \left(\frac{\partial^3F_{\rm GW}}{\partial x_i \partial x_j \partial x_k} \frac{\partial x_i}{\partial y_{\eta}}\frac{\partial x_j}{\partial y_{\delta}} \frac{\partial x_k}{\partial y_{\epsilon}} \right)g^{\eta\alpha}g^{\delta\beta}\,, 
\label{eq: lhsGW}
\end{align}
it suffices to check that  
\[\sum_{i,j,k}\frac{\partial^3F_{\rm GW}}{\partial x_i \partial x_j \partial x_k} \frac{\partial x_i}{\partial y_{\eta}}\frac{\partial x_j}{\partial y_{\delta}} \frac{\partial x_k}{\partial y_{\epsilon}}   \]
is $\widetilde{\cW}$-invariant. Since $y_\alpha$ is $\widetilde{\cW}$-invariant for $1\leqslant \alpha \leqslant l$, $y_{l+1}= x_{l+1}$, and the coordinates $x_i$ which are linearly-acted-upon by $\widetilde{\cW}$ are contracted, verifying the $\widetilde{\cW}$-invariance of the last expression amounts to checking the $\widetilde{\cW}$-invariance of the Gromov--Witten potential $F_{\rm GW}$, up to an additive quadratic shift in~$x$. Recall that 
    \beq \mathrm{Li}_3(\re^z) = c_3(z) + \mathrm{Li}_3(\re^{-z})\,, 
    \label{eq:invLi3}
    \eeq
where $c_3(z)$ is a cubic polynomial in $z$, as follows from integrating both sides of the geometric series identity $\mathrm{Li}_0(\re^z) = 1 - \mathrm{Li}_0(\re^{-z}).$ 
Expanding the sum over positive roots in \eqref{GW_potential} to a sum over all roots using in \eqref{eq:invLi3},
we have
\[F_{\rm GW} = D_3(x) + 2\nu \sum_{\beta\in \mathcal{R}} \mathrm{Li}_3(\re^{-\langle \beta,h\rangle})\,\]
for a cubic polynomial $D_3(x)$. It is straightforward to verify that $D_3(x)$ is Weyl-invariant, 
hence $F_{\rm GW}$ is Weyl-invariant. The invariance under the extended affine action is trivial, since for each $\beta$, $\re^{\langle \beta,h\rangle}$ is invariant under the corresponding translation in $\widetilde{\cW}$. \\
It only remains to verify the boundedness property. Under the limit \eqref{limit}, 
since the entries of the Jacobian matrix $\cJ_{\alpha i}$ are bounded, so are the coefficients of its inverse $\cJ^{-1}_{i \alpha}$, and we have already shown that $\Delta$ is bounded. As for the triple derivatives of the Gromov--Witten potential, boundedness is trivial when either of $i,j,k=l+1$ by the string equation and \eqref{eq:cBG}. For $i,j,k\leq l$, we have
\beq
\frac{\partial^3 F_{\rm GW}}{\partial x_{i}\partial x_{j}\partial x_{k}} = c_{ijk} + \sum_{\beta\in \RR^+} d_{ijk}(\beta)\frac{e^{-\langle  \beta, h \rangle}}{1-e^{-\langle \beta,h \rangle}}\,,
\label{eq:3derGWijk}
\eeq
for constants $c_{ijk}$, $d_{ijk}(\beta)$, with $\beta\in \RR^+$. Since $\bra \beta, \widehat{\omega}\ket>0$, we have 
\[\lim_{\tau \to +\infty} |\re^{-\bra \beta, x_0 +\widehat{\omega} \tau \ket}| =  0\,,\]
hence \eqref{eq:3derGWijk} is bounded at infinity.

As for $\ell_{\rm AW}$, by \cite{MR1606165,Brini:2021pix}, recall that when $1\leq A\leq l$, $t_A(y)$ is a polynomial in $y_1,\dots,y_l$ and $\re^{y_{l+1}}$ with polynomial inverse $y_i(t)$, and $t_{l+1} = y_{l+1}$. Moreover, $(c_{\rm AW})^{lm}_n(t)$ are polynomials in $y_1,\dots,y_l$ and $\re^{y_{l+1}}$. Therefore, \[
(\ell_{\rm AW})^{\alpha\beta}_{\epsilon}(y) \in \bbC[y_1,\dots,y_l,\re^{y_{l+1}}]\,.\] 
The claim then follows as $y_1(x), \dots, y_l(x)$ and $\re^{y_{l+1}}$ are generators of $\cA$.
\end{proof}

The square Jacobian factor $\Delta^2$ in $\Delta^2 \ell_{\rm GW}$ entered our discussion so far only in order to 
offset the potential double poles of $\ell_{\rm GW}$ along the discriminant. However, the equality \eqref{eq:AD2} that we need to prove, $\ell_{\rm GW}= \ell_{\rm AW}$, would lead to the stronger expectation that $\ell_{\rm GW}$ is in fact regular at $\delta =0$. We now show that this is indeed the case. We start by proving the following
\begin{lem}
Let
\[
\bary{rcl}
m : \Lambda_w(\RR) & \longrightarrow  & \bbC \\
\omega & \longrightarrow & m_\omega
\eary
\]
be a complex-valued map on the weight lattice having finite support and Weyl-invariant fibres,
\[
m_{w(\omega)}  = m_{\omega} \quad \forall\, w \in \cW\,, \qquad \big|\{\omega \in \Lambda_w(\RR) | m_\omega \neq 0\}\big |<\infty\,, 
\]
and consider the  Weyl-invariant Fourier polynomial
\[\mathfrak{p}(x) \coloneqq \sum_{\omega \in \Lambda_w(\RR)} m_{\omega} \re^{\langle \omega , x \rangle} \in \bbC[\re^{\pm x_1}, \dots, \re^{\pm x_l}]^\cW\,.
\]
Then, for all $\beta \in \Lambda_r(\RR)$, the directional derivative of $\mathfrak{p}(x)$ along $\beta$ is divisible  by $(1-\re^{\pm \langle\beta,h \rangle})$ in the ring of Fourier polynomials,
\[ {\sum_{a=1}^l\langle\beta,\omega_a\rangle\frac{\partial \mathfrak{p}(x)}{\partial x_a}} \in {(1-\re^{\pm \langle\beta,h \rangle})}\, \bbC[\re^{\pm x_1}, \dots, \re^{\pm x_l}]\,. \]
\label{lem: regular}
\end{lem}
\begin{proof}
Since 
\[\frac{\partial \mathfrak{p}(x)}{\partial x_a} = \sum_{\omega\in \Lambda_w(\RR)} m_{\omega} \langle \omega,\alpha^\vee_a\rangle\re^{\langle \omega , h \rangle}\,, \qquad \beta = \sum_a\langle\beta,\omega_a\rangle \alpha^\vee_a\,,\] 
we have that
\beq
\sum_{a=1}^l \langle\beta,\omega_a\rangle\frac{\partial \mathfrak{p}(x)}{\partial x_a} = \sum_{\omega\in \Lambda_w(\RR)} m_{\omega} \langle \beta,\omega\rangle\re^{\langle \omega , h \rangle}.
\label{eq: nopole1}
\eeq
Let now $\beta \in \Lambda_r(\RR)$ and
\[\Gamma_\pm(\beta) \coloneqq \l\{\omega\in \Lambda_w(\RR)\, \big|\, \mathrm{sgn}\l(\langle \beta,\omega \rangle \r) = \pm 1 \r\}\,.
\]
Since $\langle \beta,s_{\beta}(\omega)\rangle = - \langle \beta,\omega\rangle,$
the Weyl reflection across $\beta^\perp$ gives a bijection
\[s_{\beta}:\Gamma_\pm(\beta)\to \Gamma_\mp(\beta)\,.\] Moreover, using $m_{s_{\beta}(\omega)} = m_{\omega}$, we can rewrite \eqref{eq: nopole1} as
\[\sum_{\omega\in \Gamma_+(\beta)} m_{\omega} \langle \beta,\omega\rangle (\re^{\langle \omega , h \rangle} - \re^{\langle s_{\beta}(\omega) , h \rangle} ) = 
\sum_{\omega\in \Gamma_+(\beta)} m_{\omega} \langle \beta,\omega\rangle \re^{\langle \omega , h \rangle} (1- \re^{- \frac{2\langle \omega,\beta \rangle \langle \beta , h \rangle}{\langle \beta,\beta \rangle} })\,.
\]
As 
$\frac{2\langle \omega,\beta \rangle}{\langle \beta,\beta \rangle}$ is 
a (positive) integer, $(1- \re^{\pm \langle \beta , h  \rangle})$ divides the r.h.s. in $\bbC[\re^{\pm x_1}, \dots, \re^{\pm x_l}]$\,.
\end{proof}

\begin{prop}
$\l(\ell_{\rm GW}\r)^{\alpha\beta}_{\epsilon}\in \cA$.
\label{prop:lGWA}
\end{prop}
\begin{proof}
From \eqref{eq:yi}, \eqref{GW_potential} and \eqref{eq:D}, $\l(\ell_{\rm GW}\r)^{\alpha\beta}_{\epsilon}$ is a polynomial in $\re^{x_{l+1}}$. Let then
    \[
    L^{\alpha\beta}_{\epsilon} \coloneqq \l(\ell_{\rm GW}\r)^{\alpha\beta}_{\epsilon}\big|_{x_{l+1}=0}\,.
    \]
%
The statement would follow from showing that $L^{\alpha\beta}_{\epsilon}$, which {\it a priori} has double poles along $\delta=0$, is in fact a Fourier polynomial in $(x_1,\dots, x_{l})$.  By  \cref{prop:lqhom}, $\delta^2L^{\alpha\beta}_{\epsilon}$ is Weyl-invariant, and therefore $\delta L^{\alpha\beta}_{\epsilon}$ is Weyl anti-invariant. By Bourbaki \cite[Ch.~VI, Sec.3, Prop.~2(iii)]{B02}, the multiplication by $\delta$ is a bijection from the set of Weyl-invariant Fourier polynomials onto the set of Weyl anti-invariant Fourier polynomials. It then suffices to show that 
 \[\delta L^{\alpha\beta}_{\epsilon} \in \bbC[\re^{\pm x_1}, \dots, \re^{\pm x_l}]\,.
 \]
  From \eqref{eq:lhsrhs}, we have
\[ \delta L^{\alpha\beta}_{\epsilon}= 
\sum_{i,j,k,a,b}\frac{\partial^3F_{\rm GW}}{\partial x_i \partial x_j \partial x_k}(\eta^\flat_{\rm AW})^{ia} (\eta^\flat_{\rm AW})^{jb}\cJ_{\alpha a}  \cJ_{\beta b} \cJ^{-1}_{k \epsilon}\, \delta  \bigg|_{x_{l+1}=0}\,,\]
where the r.h.s. has in principle simple poles at $\delta=0$. 
Since $\Delta= \det \cJ$, $\delta  \cJ^{-1}_{k \epsilon}\bigg|_{x_{l+1}=0}$ is regular $\delta =0$, therefore it would be sufficient to show that
\[\sum_{i,a}\frac{\partial^3F_{\rm GW}}{\partial x_i \partial x_j \partial x_k}(\eta^\flat_{\rm AW})^{ia}\cJ_{\alpha a}\bigg|_{x_{l+1}=0}\]
is pole-free at $\delta =0$. By \eqref{GW_potential}, 
the triple derivatives of the prepotential are pole-free along the discriminant unless $i,j,k\leq l$, in which case\footnote{
For $\mathrm{A}_l$ with a general 2-torus action, the expression is the same upon relacing $2\nu \rightarrow \nu_1+\nu_2$.
}
\begin{align*}
    \frac{\partial^3F_{\rm GW}}{\partial x_i \partial x_j \partial x_k} 
    &=c_{ijk} + 2\nu \sum_{\beta\in \mathcal{R}^+}\langle\beta,\alpha^\vee_i\rangle \langle\beta,\alpha^\vee_j\rangle \langle\beta,\alpha^\vee_k\rangle\frac{1}{1- \re^{\langle \beta,h \rangle}}\,.
\end{align*}
By \cref{lem: regular} with $\mathfrak{p}(x)=y_\alpha(x)$, we have that, for all $\beta\in \RR^+$,
$
\sum_a\langle\beta,\omega_a\rangle \cJ_{\alpha a}
$
has a simple zero at $\bra \beta, h\ket =0$. Hence, using  
$(\eta^\flat_{\rm AW}){}_{ab}= -C_{ab}$ for $a,b,\leq l$ by \eqref{eq:inter_form}, the expression
\[\frac{1}{1-\re^{\langle\beta,h \rangle}}\sum_{i,a} \langle \beta,\alpha^\vee_i \rangle (C^{-1})_{ia}\cJ_{\alpha a}\]
is regular on the discriminant $\delta = 0$, concluding the proof.
 \end{proof}
 

\begin{prop}For any $\alpha,\beta,\epsilon$, 
    $\l(\ell_{\rm GW}\r)^{\alpha\beta}_{\epsilon}$ and $\l(\ell_{\rm AW}\r)^{\alpha\beta}_{\epsilon}$ are quasi-homogeneous polynomials in $y_1,\dots,y_l$ and $\re^{y_{l+1}}$ of degree $d_{\alpha} + d_{\beta}- d_{\epsilon}$.
    \label{prop:lqhom}
\end{prop}
\begin{proof}
The polynomiality statement follows from \cref{prop:Aring,prop:fourpol,prop:lGWA}.
The homogeneity property for $\ell_{\rm GW}$ can be read off from \eqref{eq: lhsGW}: firstly, from \eqref{eq: g20iny}, $g^{\eta\delta}$ and $g^{\delta\beta}$ are quasi-homogeneous of degree $d_{\eta} + d_{\delta}$ and $d_{\delta} + d_{\beta}$ respectively. The triple derivatives of $F_{\rm GW}$ are quasi-homogeneous of degree zero, and therefore each nonzero term of  $\l(\ell_{\rm GW}\r)^{\alpha\beta}_{\epsilon}$ has the same 
degree 
$d_{\alpha} + d_{\beta}- d_{\epsilon}.$
As for $\ell_{\rm AW}$, by \cite{MR1606165},  $t_A$ is a quasi-homogeneous polynomial of degree $d_A$ for $1\leq A\leq l$, and $ t_{l+1} = y_{l+1}$.
Thus, recalling that $d_{l+1}=0$, $\frac{\partial y_\alpha}{\partial t_A}$ is quasi-homogeneous of degree $d_\alpha - d_A$, 
and likewise $\frac{\partial t_B}{\partial y_\beta}$, has quasi-homogeneous degree $d_B - d_\beta$.
Furthermore, by \cite{MR1606165}, $(c_{\rm AW})^{LM}_N(t)$ is quasi-homogeneous of degree $d_L + d_M -d_N$. Therefore, each nonzero term of
\[(\ell_{\rm AW})^{\alpha\beta}_{\epsilon} = \sum_{L,M,N}\frac{\partial y_{\alpha}}{\partial t_L}\frac{\partial y_{\beta}}{\partial t_M}\frac{\partial t_N}{\partial y_{\epsilon}}c^{LM}_N\]
will be of the same degree: 
 \[(d_{\alpha}-d_{L})+ (d_{\beta}-d_M) +(d_N-d_{\epsilon}) + (d_L + d_M -d_N) = d_{\alpha} + d_{\beta} -d_{\epsilon}\,,\]
concluding the proof.
\end{proof}

Define a grading on $\bbC[Y_1, \dots , Y_l]$ by $\deg_Y Y_i = d_i$ for $1\leq i\leq l$. Writing
   \beq
   D \coloneqq \mathrm{max}\big\{(d_{\eta}+d_{\delta}-d_{\epsilon})|1\leq \eta,\delta,\epsilon \leq l+1\big\}
   \,,
   \label{eq:D}\eeq 
we will write \[
\mathfrak{V}_{\mathrm{adm}} \coloneqq\l\{F \in  \bbC[Y_1, \dots , Y_l]\, \Big|\, \deg_Y F \leq D\r\}\,, \quad 
S_{\mathrm{adm}} := \bigg\{(n_1,\dots, n_l)\in \mathbb{Z}_{\geqslant 0}^{l} \bigg| \sum_{i=1}^{l} n_i d_i \leq D\bigg\},\]
for, respectively, the finite-dimensional vector subspace of polynomials in $(Y_1, \dots, Y_l)$ of degree less than or equal to $D$, and the set of monic monomials in $\mathfrak{V}_{\mathrm{adm}}$.
We will refer to $S_{\mathrm{adm}}$ as the set of {\it admissible exponents} of the root system $\RR$, and to the corresponding monomial basis of $\mathfrak{V}_{\mathrm{adm}}$ as the set of {\it admissible monomials}. In particular, for $N=(n_1, \dots, n_l) \in S_{\rm adm}$, we will use the shorthand multi-index notation
\[
Y^N \coloneqq \prod_{i=1}^l Y_i^{n_i}\,.
\]

\begin{cor}
We have
\[
\l(\ell_{\rm GW}\r)^{\eta\delta}_\epsilon\big|_{x_{l+1}=0} \in \mathfrak{V}_{\mathrm{adm}}\,, \quad \l(\ell_{\rm AW}\r)^{\eta\delta}_\epsilon\big|_{x_{l+1}=0}\in \mathfrak{V}_{\mathrm{adm}} \,.\]
\label{cor:boundinY}
\end{cor}
\begin{proof}
\vspace{-.5cm}
Immediate from \cref{prop:lqhom}.
\end{proof}

By \cref{cor:boundinY}, the Weyl-invariant Fourier polynomials $(\ell_{\rm GW})^{\eta\delta}_\epsilon|_{x_{l+1}=0}$ and $(\ell_{\rm AW})^{\eta\delta}_\epsilon|_{x_{l+1}=0}$ are elements of the same finite-dimensional vector space $\mathfrak{V}_{\mathrm{adm}}$, with $\dim_\bbC \mathfrak{V}_{\mathrm{adm}} = |S_{\rm adm}|$. This reduces the 
verification of the functional relation \eqref{eq:AD2} to checking the \emph{numerical} relation \[(\ell_{\rm GW})^{\eta\delta}_\epsilon(x^{(K)})=(\ell_{\rm AW})^{\eta\delta}_\epsilon(x^{(K)})\] for {\it finitely many points} $\{x^{(1)}, \dots,  x^{(|S_{\rm adm}|)} \}$ in $\mathfrak{h}$, each
providing a linear constraint on the coefficients of either side of \eqref{eq:AD2} as an element of $\mathfrak{V}_{\rm adm}$. If these points are chosen generically, the resulting linear system will have maximal rank and determine  the $(2,1)$ tensors $\ell_{\rm GW}$ and $\ell_{\rm AW}$ uniquely.

\begin{defn}
    A set of $|S_{\rm adm}|$ points $\{x^{(1)}, \dots, x^{(|S_{\rm adm}|)}\}\subset M^{\rm ss}_{\rm AW}$ will be called a \emph{set of initial conditions} for $\cM_{\rm AW}^\flat$ if the
$|S_{\rm adm}|\times |S_{\rm adm}|$
generalised Vandermonde matrix minor \[(Y^{N}(x^{(M)}))_{N,M\in S_{\rm adm}}\] is non-singular,
\beq
\det_{M,N} \Big(Y^{N}(x^{(M)})\Big) \neq 0\,. 
\label{eq:vdm}
\eeq
\end{defn}

Checking \eqref{eq:AD2} on a set of initial conditions is poorly suited to a general proof of \cref{thm:ddQH} for all marked pairs $(\RR, \widehat{\omega})$. On the other hand, it can be performed highly effectively on a case-by-case basis. As it just remains to prove  \cref{thm:msQH,thm:ddQH} for $\RR=\mathrm{E}_l$, we will construct sets of initial conditions to verify \eqref{eq:AD2} directly in these three exceptional cases.

\subsubsection{Proof of \cref{thm:ddQH} for the $\mathrm{E}_l$ series}

It is straightforward to compute the degree bound $D$ and the number $|S_{\rm adm}|$ of admissible monomials for the exceptional series $\mathrm{E}_l$. These are reproduced in \cref{tab:DSadm}.

\begin{table}[h]
\centering
\begin{tabular}{|c|c|c|}
\hline
$l$ & $D$ & $|S_{\rm adm}|$\\
\hline
6 & 12 & 151 \\
\hline
7 & 24 & 254 \\
\hline
8 & 60 & 434\\
\hline
\end{tabular}
\medskip

\caption{Degree bounds and size of the set of initial conditions for the exceptional series $\mathrm{E}_l$, $l=6,7,8$.}
\label{tab:DSadm}
\end{table}

Configurations of $|S_{\rm adm}|$ points that are not initial have measure zero in $\mathrm{Sym}^{|S_{\rm adm}|}\mathfrak{h}$,
since \eqref{eq:vdm} is an open condition. 
Due to the small size of $S_{\rm adm}$ for $\RR =\mathrm{E}_l$, it is straightforward to construct an initial set by picking a configuration of $|S_{\rm adm}|$ points in the Cartan subalgebra and then checking {\it a posteriori} that the generalised Vandermonde minor \eqref{eq:vdm} is indeed non-zero for them. Having constructed a set of initial conditions $\mathfrak{I}$, we are then just left with verifying directly the numerical identities
\[
(\ell_{\rm GW})^{\alpha\beta}_\gamma(x^{(K)}) = 
(\ell_{\rm AW})^{\alpha\beta}_\gamma(x^{(K)})\,, \quad k=1, \dots, |\mathfrak{I}|\,.
\] 
\begin{prop}
For all $(\RR, \widehat{\omega})$ with $\RR=\mathrm{E}_l$, there exists a set of initial conditions $\mathfrak{I}$ such that, $\forall x \in \mathfrak{I}$,
\beq
(\ell_{\rm GW})^{\alpha\beta}_\gamma(x) = 
(\ell_{\rm AW})^{\alpha\beta}_\gamma(x)\,.
\label{eq:ddinit}
\eeq
\label{prop:ddinit}
\end{prop}
\begin{proof} 
\vspace{-.75cm}
Direct calculation.
\end{proof}

\begin{example}[$\RR=\mathrm{E}_6$] We will sketch here the main elements entering the verification of \eqref{eq:ddinit} for $\RR=\mathrm{E}_6$. The l.h.s. of \eqref{eq:AD2}, as a function of $x=(x_1, \dots, x_{7})$ is explicitly computed by \eqref{GW_potential}. As for the r.h.s., $(\eta^\flat_{\rm AW})_{ia}$ is given by  \eqref{eq:inter_form}, and all we need to compute are the Jacobian matrix of the change-of-variables $t \rightarrow t(x)$, its inverse, and the $(2,1)$ multiplication tensor $c_{\rm AW}$ in flat coordinates. For the latter, the prepotential of $\cM_{\rm AW}$ was computed in 
\cite{Brini:2021pix} to be
\begin{align*}
& 
F_{\rm AW}  =\frac{\re^{12 t_7}}{24}+
\frac{1}{4} \re^{8 t_7} t_1 t_5+\frac{1}{12} \re^{6 t_7} \left(t_1^3+t_5^3+3 t_6^2\right)+\frac{1}{2} \re^{5 t_7} t_1 t_6 t_5 
\nn \\ &
+\frac{1}{72} \re^{4 t_7} \left(10 t_1^2 t_5^2+t_1^2 t_4 -t_2 t_5^2-t_2t_4\right)+\frac{\re^{3 t_7}}{12}  t_6 \left(t_1^3-t_2 t_1+t_5^3+2 t_6^2+t_4
   t_5\right)\nn \\ & +\frac{\re^{2 t_7}}{144}  \left(t_5 t_1^4-2 t_2 t_5 t_1^2+\left(t_5^4+2 t_4 t_5^2+36 t_6^2 t_5+t_4^2\right) t_1+t_2^2
   t_5\right) +\frac{1}{72} \re^{t_7} \left(t_1^2-t_2\right) \left(t_5^2+t_4\right) t_6 \nn \\ 
   & -\frac{t_1^6+t_5^6}{38880}-\frac{t_2 t_1^4-t_4 t_5^4+t_2^3-t_4^3}{7776}-\frac{t_2^2 t_1^2+t_4^2 t_5^2}{2592}-\frac{(t_2 t_1-t_4 t_5)t_3 }{36} -\frac{t_6^4}{192} +\frac{1}{8} t_3 t_6^2+\frac{1}{4} t_3^2 t_7\,,
\end{align*}
from which 
$(c_{\rm AW})^{AB}_C$ in \eqref{eq:cflat}
is immediately computed as an explicit polynomial in $t_1, \dots, t_6$ and $\re^{t_7}$. The flat coordinates are related to the fundamental traces in \eqref{eq:Wi} as \cite{Brini:2021pix}
\[
      t_1 = W_0^{\frac{1}{3}}W_1\,, \quad  t_2 = W_0^{\frac{2}{3}}(W_1^2 - 6W_2  - 12W_5)\,, \quad  t_3 = W_0(2W_1W_5 + W_3 + 3W_6 + 3)\,,  \]
\beq
    t_4 = W_0^{\frac{2}{3}}(-W_5^2  + 12W_1 + 6W_4)\,, \quad  t_5 = W_0^{\frac{1}{3}}W_5\,, \quad  t_6 = W_0^{\frac{1}{2}}(W_6 + 2)\,, \quad t_7 = \frac{\log(W_0)}{6}\,, 
\label{eq:ttow}
\eeq
where furthermore $W_0 = \re^{1/2 x_{7}}$. The fundamental traces $W_1$ and $W_6$ in, respectively, the {\bf 27} and {\bf 78} (adjoint) representation can be computed from the respective weight and root system, and their expressions are given in the 
ancillary online material.
%
%
The remaining traces can be computed from the following relations in the representation ring of $\mathrm{E}_6$:
\[
W_i(x)=W_{6-i}(-x)\,, \quad i=1, \dots, 5\]
\[
W_2(x)=\frac{1}{2}\l(W_1^2(x)-W_1(2x)\r)\,, \quad W_3(x)= \frac{1}{3}\l(W_2(x) W_1(x)-W_1(x) W_1(2x)+W_1(3x)\r)
\]
expressing, respectively, the fact that $\rho_i= \overline{\rho}_{6-i}$ for $1 \leq i \leq 5$ and $\rho_i= \wedge^i \rho_1$ for $i=2,3$. 
Plugging the resulting expressions into \eqref{eq:ttow} gives the change-of-variables $t \to t(x)$, from which the Jacobian coefficients $\de_{x_i} t_A$ could then be in principle be computed as explicit, if cumbersome, Fourier polynomials with a few thousand terms. The resulting matrix inversion computing $\de_{t_A} x_i$ as rational trigonometric functions of $x$ is far out of reach of modern symbolic computation packages. On the other hand, evaluating on a (rational) set of initial conditions $\mathfrak{I}$ dramatically reduces the unwieldy expressions above in the field $\bbQ(\re^{x_1}, \dots, \re^{x_7})$ to eminently manageable manipulations of rational numbers. The exact inversion of the \emph{numerical} $7\times 7$ matrix $\de_{t_A} x_i|_{x=x^{(K)}}$ over $\bbQ$ for $x^{(K)} \in\mathfrak{I}$ takes now a fraction of a  second in {\it Mathematica} on an entry-level desktop computer\footnote{Absolute clock-times based on a setup with Intel Core i7-8700 @3.20 GHz processor and 16GB RAM.}, as does the evaluation of all the other quantities entering \eqref{eq:AD2}. This allows for a straightforward and fast verification of \eqref{eq:AD2} over $\mathfrak{I}$, and therefore, by \cref{prop:lqhom,cor:boundinY}, on the whole of $M_{\rm AW}$. 
\end{example}

The same \emph{Mathematica} calculations take a couple of minutes for $\RR=\mathrm{E}_7$, and a few hours for $\RR=\mathrm{E}_8$ with the same setup. The Wolfram Language code used to verify \cref{prop:ddinit} is available to the reader in the ancillary online material.


\begin{cor}
\cref{thm:ddQH,thm:msQH} hold for $(\RR, \widehat{\omega})$ with $\RR=\mathrm{E}_l$.
\end{cor}
\begin{proof}
\cref{thm:ddQH} follows from \cref{prop:ddinit,prop:lqhom,cor:boundinY}, and it implies \cref{thm:msQH} by \cref{thm:mirrorEAW}.
\end{proof}

\begin{example}[Landau--Ginzburg mirror symmetry for the $\mathrm{E}_6$ du Val resolution]
Let $\RR=\mathrm{E}_6$, so that $\widehat{\omega}=\omega_3$ is the highest weight of its 2925-dimensional fundamental representation. By \cref{thm:msQH} and \eqref{eq:phiresc}, the $\mathsf{T}$-equivariant quantum cohomology of the du~Val resolution of type $\mathrm{E}_6$ is mirror to a one-dimensional Landau--Ginzburg model with a log-meromorphic superpotential $\log \lambda$ and logarithmic primitive form  
\[\phi^2 = \frac{\nu}{3} \l(\frac{\rd \mu}{ \mu}\r)^2 \,,\]
where $\lambda$ and $\mu$ are the two Cartesian projections on the family of plane algebraic curves over $\mathfrak{h}^{\rm reg} \oplus \bbC$ in \eqref{eq:shiftfun} given by \[\cP(y_1, \dots y_7; \lambda, \mu)=0\,,\] 
and we took $\omega \coloneqq \omega_1$ in \eqref{eq:shiftfun} to be the highest weight of the 27-dimensional fundamental representation $\rho_1$. Expressing the characters of the exterior powers of $\rho_1$ in terms of fundamental characters $W_1, \dots W_6$ (see \cite[Eq.~(6.21)]{Borot:2015fxa}), and further relating the latter to the basic invariants as
\[
W_1 = Y_1\,, \quad W_2 = Y_2 +5 Y_5\,, \quad W_3= Y_3 + 4 Y_1 Y_5 - 9 Y_6-63\,,\]
\[ W_4 = Y_4 + 5 Y_1\,, \quad W_5= Y_5\, \quad W_6= Y_6+6\,,
\]
as can be ascertained from a direct inspection of the fundamental weight systems $\Gamma_i$ ($i=1 \dots 6$), we compute from \eqref{eq:shiftfun}--\eqref{eq:phiresc} that
\[\cP(y_1, \dots, y_7; \lambda, \mu) =  \cQ(Y_1, \dots, Y_6, \mu) + \re^{-x_7/(2 \nu)} \cQ^{[1]}(Y_1, \dots, Y_6, \mu) \lambda+ \re^{-x_7/\nu}\mu ^9 \left(\mu ^3-1\right)^3 \lambda^2\bigg|_{Y_\alpha=y_\alpha \re^{-d_\alpha y_7}}
\]
where $\cQ$ is the characteristic polynomial in \eqref{eq:weylrel}, and
\begin{align*}
 \cQ^{[1]}  & = 
1-2 \mu ^2 Y_1+\mu ^3 \left(3 Y_6+20\right)+\mu ^4 \left(Y_1^2-2 Y_2-13 Y_5\right)-\mu ^5 \left(Y_6 Y_1+9 Y_1+Y_4\right) \nn \\ 
&+ \mu ^6 \left(2 Y_3+12 Y_1 Y_5-21 Y_6-150\right)
  +\mu ^7
   \left(5 Y_2-10 Y_1^2-Y_4 Y_1+24 Y_5-Y_5 Y_6\right)  \nn \\ & + \mu ^8 \left(Y_5 Y_1^2+3 Y_6 Y_1+26 Y_1-13 Y_5^2+Y_4-2 Y_2 Y_5\right)  
   \nn \\ &
   +\mu ^9 \left(6 Y_5 Y_1-Y_1^3+3 Y_2 Y_1-6 Y_3+3 Y_4 Y_5+48
   Y_6+343\right)     \nn \\ & +\mu ^{10} \left(4 Y_5-4 Y_1^2-2 Y_5^2 Y_1-Y_4 Y_1-3 Y_2+3 Y_5 Y_6\right)
    \nn \\
   & + \Big(\mu^k \longrightarrow (- \mu)^{21-k},\, Y_1 \longleftrightarrow Y_5,\, Y_2 \longleftrightarrow Y_4  \Big)\,.
\end{align*}
As a check, from \eqref{eq:todalogeta}, the dual pairing on $M_{\rm AW} \simeq M_{\rm LG}$ is
\begin{align*}
\eta_{\rm LG}^\flat(\de_{x_i}, \de_{x_j}) & =  \frac{1}{6}\sum_m \Res_{p_m^{\rm cr}} \frac{\delta_{\de_{x_i}} \lambda\, \delta_{\de_{x_j}} \lambda}{\lambda \mu\, \partial_\mu \lambda } \frac{\rd \mu}{\mu} = -\frac{1}{6}\sum_{\lambda(p), \mu(p) \in \{0,\infty\}} \Res_{p} \frac{\delta_{\de_{x_i}}\lambda\, \delta_{\de_{x_j}}\lambda}{\lambda \mu\, \partial_\mu \lambda } \frac{\rd \mu}{\mu}\,.
\end{align*}
For $1\leq i,j \leq l$, it is straightforward to check that the only non-vanishing residues arise from the zeroes of $\lambda$, i.e. when $\mu=\re^{\bra \omega', x \ket}$ for $\omega' \in \Gamma_1$. We find
\[
\sum_{\omega' \in \Gamma_1} \Res_{\mu= \re^{\bra \omega', x \ket}} \frac{\delta_{\de_{x_i}}\lambda\, \delta_{\de_{x_j}}\lambda}{\lambda \mu\, \partial_\mu \lambda } \frac{\rd \mu}{\mu} = 
\sum_{\omega' \in \Gamma_1} \bra \omega', \alpha_i \ket \bra \omega', \alpha_j \ket = \left(
\begin{array}{cccccc}
 12 & -6 & 0 & 0 & 0 & 0 \\
 -6 & 12 & -6 & 0 & 0 & 0 \\
 0 & -6 & 12 & -6 & 0 & -6 \\
 0 & 0 & -6 & 12 & -6 & 0 \\
 0 & 0 & 0 & -6 & 12 & 0 \\
 0 & 0 & -6 & 0 & 0 & 12 \\
\end{array}
\right)_{ij},
\]
and therefore
\[
\eta_{\rm LG}^\flat(\de_{x_i}, \de_{x_j}) = \eta_{\rm GW}(\varphi_i, \varphi_j) = -C_{ij}\,, 
\]
recovering the expression for the 2-point intersection pairing on $Z$ in \eqref{eq:cBG}.
\end{example}

 \bibliographystyle{plain}
\bibliography{EAWduValMirror}
\end{document}